\documentclass[12pt]{article}

\usepackage{amssymb}
\usepackage{amsthm}
\usepackage{fleqn}
\usepackage{graphicx}
\usepackage{epstopdf}
\usepackage{amsfonts}
\usepackage[numbers]{natbib}
\usepackage[colorlinks,citecolor=blue,urlcolor=blue]{hyperref}
\usepackage{amsmath}
\usepackage{geometry}
\usepackage{array}
\usepackage{tabularx}

\newtheorem{theorem}{Theorem}[section]
\newtheorem{lemma}[theorem]{Lemma}

\newtheorem{corollary}[theorem]{Corollary}
\newtheorem{remark}[theorem]{Remark}
\newtheorem{definition}[theorem]{Definition}

\renewcommand\theequation{\arabic{section}.\arabic{equation}}
\makeatletter\@addtoreset{equation}{section}\makeatother\makeatletter

\usepackage{stmaryrd}
\usepackage{subfig}
\usepackage{booktabs}

\begin{document}

\title{An infinite horizon sufficient stochastic maximum principle for regime switching
diffusions and applications\thanks{This work was supported by the National Natural Science Foundation
of China (12471414, 12171086), the Natural Science Foundation of Jiangsu Province (BK20242023),
the Jiangsu Province Scientific Research Center of Applied Mathematics (BK20233002),
and the Research Grants Council of Hong Kong (15226922, 15225124, PolyU 4-ZZVB, PolyU 4-ZZP4, PolyU 1-ZVXA).}}

\author{
Kai Ding\thanks{School of Mathematics, Southeast University,
Nanjing 211189, China (dkaiye@163.com).}
\and
Xun Li\thanks{Department of Applied Mathematics,
The Hong Kong Polytechnic University, Kowloon, Hong Kong, China
(li.xun@polyu.edu.hk).}
\and
Siyu Lv\thanks{School of Mathematics, Southeast University,
Nanjing 211189, China (lvsiyu@seu.edu.cn).}
\and
Xin Zhang\thanks{School of Mathematics, Southeast University,
Nanjing 211189, China (x.zhang.seu@gmail.com).}
}

\date{}

\maketitle

\begin{abstract}
This paper is concerned with a \emph{discounted} stochastic optimal control problem
for regime switching diffusion in an infinite horizon. First, as a preliminary with
particular interests in its own right, the global well-posedness of infinite horizon forward
and backward stochastic differential equations with Markov chains and the \emph{asymptotic property}
of their solutions when time goes to infinity are obtained. Then, a sufficient stochastic
maximum principle for optimal controls is established via a dual method under certain convexity
condition of the Hamiltonian. As an application of our maximum principle, a linear quadratic
production planning problem is solved with an \emph{explicit} feedback optimal production rate.
The existence and uniqueness of a \emph{non-negative} solution to the associated algebraic
Riccati equation are proved. Numerical experiments are reported to illustrate the theoretical
results, especially, the \emph{monotonicity} of the value function on various model parameters.
\end{abstract}


\section{Introduction}

The regime switching model has gained considerable attention in recent years, especially in the fields of
finance and economics, due to their superior empirical performance compared to traditional diffusion models.
For example, Zhang \cite{zhang2001stock} developed an optimal stock selling rule by using a regime switching
geometric Brownian motion model, which can capture different appreciation and volatility rates of the stock
under different market treads. Hardy \cite{hardy2001regime} further demonstrated that the regime switching
model fits real financial market much better than other conventional models based on analysis of the data
from Standard \& Poor's 500 and Toronto Stock Exchange 300 indexes. In the past few decades, extensive
research on optimal control theory and applications has been conducted for regime switching models,
such as linear quadratic (LQ for short) optimal control (see Li and Zhou \cite{LiZhou2002}
and Zhang et al. \cite{zhang2021}), mean-variance portfolio selection (see Zhou and Yin \cite{zhou2003}
and Hu et al. \cite{hu2022}), and European and American option pricing (see Yao et al. \cite{YaoZhangZhou2006}
and Buffington and Elliott \cite{BE2002}).

Stochastic maximum principle (SMP for short) is well known as one of the fundamental methods
for solving stochastic optimal control problems. Among others, we would like to review the
following papers that are closely relevant to our work (i.e., in a regime switching model).
In the spirit of Zhou \cite{Zhou1996TAC}, Donnelly \cite{donnelly2011} proposed a sufficient
SMP for regime switching diffusion. Nguyen et al. \cite{nguyen2020,nguyen2021} obtained local
and general SMPs for optimal control problems with both regime switching and mean-field interactions.
Lv et al. \cite{LXX2023SCL} proved a partial information SMP when the underlying Markov chain
can not be observed. Note that these results are all based on a finite horizon. On the other hand,
there is relatively less literature on SMP in an infinite horizon. In view of this fact, this paper
is concerned with a sufficient SMP for optimal control problem with regime switching in an
\emph{infinite horizon} and \emph{discounted formulation}
(see also Maslowski and Veverka \cite{maslowski2014b} and Zheng and Shi \cite{ZhengShi2023}
for the case with no regime switching).

To establish an SMP in an infinite horizon, it is crucial to address the global well-posedness and
asymptotic property of infinite horizon stochastic Hamiltonian system, which is a kind of infinite
horizon forward-backward stochastic differential equation (FBSDE for short). Typically, there are
two main frameworks to deal with this issue: one imposes the so-called $L^2$-stabilizability condition
(see Sun and Yong \cite{SunYong2018}) to guarantee that $\mathbb{E}[\int_0^{\infty}|\varphi_t|^2dt]<\infty$
and $\lim_{T\rightarrow\infty}\mathbb{E}[|\varphi_T|^2]=0$ for the solution $\varphi=X,Y$, while the other
one introduces a negative parameter $K$ to suppress the growth of $X$ and $Y$ through an exponentially
discounted term $e^{2Kt}$, which leads to that $\mathbb{E}[\int_0^{\infty}e^{2Kt}|\varphi_t|^2dt]<\infty$
and $\lim_{T\rightarrow\infty}\mathbb{E}[e^{2KT}|\varphi_T|^2]=0$ for $\varphi=X,Y$
(see Peng and Shi \cite{peng2000infinite}, Yu \cite{yu2017a}, and Wei and Yu \cite{wei2021}).
Benefiting from the presence of a discount factor $r$ in the cost functional of our discounted formulation,
this paper naturally adopts the latter framework.

The problem of \emph{production planning} has been an important topic of great interests in management
and manufacturing. The seminal work can be traced back to Thompson and Sethi \cite{thompson1980}, where
the goal is to minimize a quadratic loss function, which measures the deviation of \emph{inventory}
and \emph{production} from their target levels, and an optimal linear decision rule was derived by using
the maximum principle of \emph{deterministic} control problems. Then, Bensoussan et al. \cite{bensoussan1984}
and Fleming et al. \cite{fleming1987} studied \emph{stochastic} production planning problems by analyzing
the corresponding Hamilton-Jacobi-Bellman equations. However, existing literature on this subject exhibits
limitations in capturing random switching among different market regimes. To make up this gap,
in this paper we suppose that the market parameters are modulated by a finite-state Markov chain
and employ our sufficient SMP to obtain an optimal production rate.

The main contributions of this paper include the following three aspects:

(\romannumeral1) We establish the global well-posedness of infinite horizon SDEs and BSDEs with Markov chains.
In the meanwhile, we prove the asymptotic property of their solutions at infinity; such a result is indispensable
for us to establish the sufficient SMP. Moreover, we determine a suitable \emph{range} of the discount factor $r$,
which ensures the existence and uniqueness of solutions to the associated stochastic Hamiltonian system.

(\romannumeral2) We present a sufficient SMP for the discounted optimal control problem with regime switching
in an infinite horizon. In particular, with the help of the asymptotic property of solutions to the
stochastic Hamiltonian system, we do not need to additionally assume the \emph{transversality condition}
imposed by Haadem et al. \cite{haadem2013a} to prove the SMP.

(\romannumeral3) We apply our sufficient SMP to solve a production planning problem in a regime switching
market, which turns out to be, essentially, an infinite horizon LQ optimal control problem. Based on the
so-called \emph{four-step scheme}, we decouple the stochastic Hamiltonian system and get an algebraic
Riccati equation (ARE for short). Utilizing matrix analysis and a kind of \emph{elimination method}, we prove
the existence and uniqueness of a \emph{nonnegative} solution to the ARE. Then, an \emph{explicit} optimal
production rate is obtained, which has a feedback form of the inventory process and the market regime.
A set of numerical tests are also provided to display the effects of regime switching and the \emph{monotonicity}
of the value function on various model parameters.

The rest of this paper is organized as follows. Section \ref{section PF} formulates the discounted stochastic
optimal control problem under consideration. Section \ref{section equations} deals with the unique solvability and
asymptotic property of infinite horizon SDEs and BSDEs with Markov chains. Section \ref{section SMP} establishes
the sufficient stochastic maximum principle. Section \ref{section production planning} applies the SMP to solve
a production planning problem and reports some numerical results. Finally, Section \ref{section conclusion}
concludes the paper with some further remarks.

\section{Problem formulation}\label{section PF}

Let $(\Omega,\mathcal{F},\mathbb{P})$ be a complete probability space with the natural filtration
$\mathbb{F}=\{\mathcal{F}_t\}_{t\geq0}$ generated by the following two \emph{mutually independent}
stochastic processes: a one-dimensional standard Brownian motion $W_t$, $t\geq0$, and
a finite-state Markov chain $\alpha_t$, $t\geq0$, taking values in $\mathcal{M}=\{1,\ldots,m\}$
with generator (i.e., the matrix of transition rates) $Q=(q_{ij})_{i,j\in\mathcal{M}}$.

For each $i,j=1,\ldots,m$ with $i\ne j$, let $N_{ij}(t)$ be the number of jumps of the chain
$\alpha$ from state $i$ to state $j$ up to time $t$, i.e.,
$$
N_{ij}(t)=\sum_{0<s\leq t}1_{\{\alpha_{s-}=i\}}1_{\{\alpha_{s}=j\}},
$$
where $1_{A}$ denotes the zero-one indicator function of a set $A$. Note that $N_{ij}(t)$
is a counting process, whose \emph{optional} and \emph{predictable} quadratic variation processes
are given by
$$
[N_{ij}](t)=\sum_{0<s\leq t}{|\Delta N_{ij}(s)|^2}=\sum_{0<s\leq t}{|\Delta N_{ij}(s)|}=N_{ij}(t),
$$
and
$$
\langle N_{ij}\rangle(t)=\int_0^t q_{ij}1_{\{\alpha_{s-}=i\}}ds,
$$
respectively. It follows from Nguyen et al. \cite{nguyen2020,nguyen2021} that the process
$M_{ij}(t)$, $t\geq0$, defined by
$$
M_{ij}(t)=[N_{ij}](t)-\langle N_{ij}\rangle(t),
$$
is a purely discontinuous and $\mathbb{F}^{\alpha}$-square integrable martingale,
where $\mathbb{F}^{\alpha}=\{\mathcal{F}_t^{\alpha}\}_{t\geq0}$ is the natural
filtration of the Markov chain $\alpha$. Furthermore,
$$
[M_{ij}](t)=[N_{ij}](t)=N_{ij}(t),\quad
\langle M_{ij}\rangle(t)=\langle N_{ij}\rangle(t)=\int_0^t q_{ij}1_{\{\alpha_{s-}=i\}}ds.
$$
For any given $\mathbb{F}$-adapted process $\Lambda=\{\Lambda_{ij}\}_{i,j\in\mathcal{M}}$
with $\Lambda_{ij}$ taking values in $\mathbb{R}^n$, we denote
$$
\Lambda_s\bullet dM_s=\sum_{i,j\in\mathcal{M}}\Lambda_{ij}(s)dM_{ij}(s),\quad
\int_0^t \Lambda_s\bullet dM_s=\sum_{i,j\in\mathcal{M}}\int_0^t \Lambda_{ij}(s)dM_{ij}(s),
$$
and
$$
\Lambda_s\bullet d[M]_s=\sum_{i,j\in\mathcal{M}}\Lambda_{ij}(s)d[M_{ij}](s),\quad
\int_0^t \Lambda_s\bullet d[M]_s=\sum_{i,j\in\mathcal{M}}\int_0^t \Lambda _{ij}(s)d[M_{ij}](s).
$$
Let $T\in(0,\infty)$ and $K\in\mathbb{R}$. We introduce the following two spaces of processes
which will be used frequently in this paper:
$L_{\mathbb{F}}^{2,K}(0,\infty;\mathbb{R}^n)$ denotes the space of all $\mathbb{R}^n$-valued
$\mathbb{F}$-adapted processes $X$ such that
$$
\mathbb{E}\bigg[\int_0^{\infty}e^{2Kt}|X_t|^2dt\bigg]<\infty,
$$
and $M_{\mathbb{F}}^{2,K}(0,\infty;\mathbb{R}^n)$ denotes the space of all $\Lambda=\{\Lambda_{ij}\}_{i,j\in\mathcal{M}}$
with each $\Lambda_{ij}$ is an $\mathbb{R}^n$-valued $\mathbb{F}$-adapted process such that
$$
\mathbb{E}\bigg[\int_0^{\infty}e^{2Kt}|\Lambda_t|^2\bullet d[M]_t\bigg]
=\sum_{i,j\in\mathcal{M}}\mathbb{E}\bigg[\int_0^{\infty}e^{2Kt}|\Lambda_{ij}(t)|^2d[M_{ij}](t)\bigg]<\infty.
$$
%In particular, we shall abbreviate the above two notations as
%$L_{\mathbb{F}}^{2}(0,\infty;\mathbb{R}^n)$ and $M_{\mathbb{F}}^{2}(0,\infty;\mathbb{R}^n)$,
%respectively, when $K=0$.

We now introduce the discounted optimal control problem that will be studied in this paper.
Consider the following state equation:
\begin{equation}\label{FSDE}
\left\{
\begin{aligned}
dX_{t}=&b(X_{t},\alpha_{t},u_{t})dt+\sigma(X_{t},\alpha_{t},u_{t})dW_{t},\quad t\geq0,\\
X_{0}=&x\in \mathbb{R}^n,\quad\alpha_{0}=i\in\mathcal{M},
\end{aligned}
\right.
\end{equation}
where $X\in\mathbb{R}^{n}$ is the state process, $u\in\mathbb{R}^{k}$ is the control process,
and $b,\sigma:\mathbb{R}^{n}\times\mathcal{M}\times\mathbb{R}^{k}\mapsto\mathbb{R}^{n}$ are two given functions.

The cost functional is given by
\begin{equation*}
%\left\{
\begin{aligned}
J(x,i;u)=\mathbb{E}\int_{0}^{\infty}e^{-rt}f(X_{t},\alpha_{t},u_{t})dt,
\end{aligned}
%\right.
\end{equation*}
where $r>0$ is the discount factor and $f:\mathbb{R}^{n}\times\mathcal{M}\times\mathbb{R}^{k}\mapsto\mathbb{R}$
is a given function.

The optimal control problem is to find an element $u^*\in\mathcal{U}_{ad}$ such that
$$
v(x,i)=\inf_{u\in\mathcal{U}_{ad}}J(x,i;u)=J(x,i;u^*),
$$
where the \emph{admissible control set} $\mathcal{U}_{ad}$ is defined as
$$
\mathcal{U}_{ad}=L_{\mathbb{F}}^{2,-\frac{r}{2}}(0,\infty;\mathbb{R}^k).
$$
In this paper, we aim to establish a sufficient condition for an optimal control in $\mathcal{U}_{ad}$.
We impose the following assumptions on the coefficients $(b,\sigma,f)$.

(H1) The functions $b$ and $\sigma$ are uniformly Lipschitz continuous with respect to $x$,
i.e., there exist constants $\kappa_{b}>0$ and $\kappa_{\sigma}>0$ such that
$$
|b(x,i,u)-b(\overline{x},i,u)|\leq\kappa_{b}|x-\overline{x}|,
\quad
|\sigma(x,i,u)-\sigma(\overline{x},i,u)|\leq\kappa_{\sigma}|x-\overline{x}|.
$$
In addition, $b(0,i,u)$ and $\sigma(0,i,u)$ have a linear growth with respect to $u$,
i.e., there exists a constant $C>0$ such that
$$
|b(0,i,u)|+|\sigma(0,i,u)|\leq C(1+|u|).
$$

(H2) The function $b$ satisfies a monotonic condition with respect to $x$ in the sense that
there exists a constant $\kappa_1\in\mathbb{R}$ such that
$$
\langle b(x,i,u)-b(\overline{x},i,u),x-\overline{x}\rangle\leq-\kappa_1 |x-\overline{x}|^2.
$$

(H3) The function $f$ has a quadratic growth with respect to $(x,u)$,
i.e., there exists a constant $C>0$ such that
$$
|f(x,i,u)|\leq C(1+|x|^2+|u|^2).
$$

\section{Infinite horizon (B)SDEs with Markov chains}\label{section equations}

\subsection{Infinite horizon SDEs with Markov chains}\label{subsection SDEs}

Under Assumption (H1), we know that (see Shen and Siu \cite{shenMaximumPrincipleJumpdiffusion2013})
on any finite horizon $[0,T]$ and $u\in L_\mathbb{F}^2(0,T;\mathbb{R}^k)$, the SDE (\ref{FSDE})
admits a unique strong solution $X$ such that $\mathbb{E}[\sup_{0\leq t\leq T}|X_t|^2]<\infty$.
However, we cannot further obtain that $\mathbb{E}[\int_0^{\infty}|X_t|^2dt]<\infty$ in the infinite
horizon. The following theorem establishes the global solvability of infinite horizon SDEs with
Markov chains in the space $L_{\mathbb{F}}^{2,K}(0,\infty;\mathbb{R}^n)$ with an appropriate $K$.
In the sequel, we shall denote by $C$ a generic constant, which may vary from line to line.
\begin{theorem}\label{thm:X-wellpose}
Let Assumptions (H1)-(H3) hold. For any $u\in L_{\mathbb{F}}^{2,K}(0,\infty;\mathbb{R}^k)$,
the SDE (\ref{FSDE}) admits a unique strong solution $X\in L_{\mathbb{F}}^{2,K}(0,\infty;\mathbb{R}^n)$
with $K<(\kappa_1-\frac{1}{2}\kappa_{\sigma}^2)\wedge 0$. Moreover, for any $\varepsilon>0$, we have
\begin{equation*}
%\left\{
\begin{aligned}
(2\kappa_1-2K-\kappa_{\sigma}^{2}-2\varepsilon)\mathbb{E}\bigg[\int_0^{\infty}e^{2Ks}|X_s|^2ds\bigg]
\leq |x|^2+\frac{C}{\varepsilon}\mathbb{E}\bigg[\int_0^{\infty}e^{2Ks}(1+|u_s|^2)ds\bigg],
\end{aligned}
%\right.
\end{equation*}
and
\begin{equation}\label{Xlim}
\lim_{T\rightarrow\infty}\mathbb{E}[|X_T e^{KT}|^2]=0.
\end{equation}
\end{theorem}
\begin{proof}
Applying It\^{o}'s formula to $e^{2Ks}|X_s|^2$ over $[t,T]$ gives
\begin{equation}\label{X-Ito}
\begin{aligned}
|X_T e^{KT}|^2=&|X_t e^{Kt}|^2+\int_t^T e^{2Ks}[2K|X_s|^2+2\langle X_s,b(X_s,\alpha_s,u_s)\rangle
+|\sigma(X_s,\alpha_s,u_s)|^2]ds\\
&+\int_t^T e^{2Ks}\langle X_s,\sigma(X_s,\alpha_s,u_s)\rangle dW_s.
\end{aligned}
\end{equation}
%We claim that $\int_0^{\cdot}e^{2Ks}\langle X_s,\sigma(X_s,\alpha_s,u_s)\rangle dW_s$
%is a uniformly integrable martingale. Indeed, by BDG inequality and Assumption (H1.1),
%we obtain
%\begin{equation*}
%\begin{aligned}
%&\mathbb{E}\left[\sup_{0\leq t\leq T}
%\left|\int_0^t e^{2Ks}\langle X_s,\sigma(X_s,\alpha_s,u_s)\rangle dW_s\right|\right]\\
%\leq&C\mathbb{E}\left[\left(\int_0^T|e^{2Ks}\langle X_s,\sigma(X_s,\alpha_s,u_s)\rangle|^2ds\right)^{\frac{1}{2}}\right]\\
%\leq&C\mathbb{E}\left[\sup_{0\leq t\leq T}|e^{2Kt}X_t|
%\left(\int_0^T e^{2Ks}|\sigma(X_s,\alpha_s,u_s)|^2ds\right)^{\frac{1}{2}}\right]\\
%\leq&C\mathbb{E}\left[\sup_{0\leq t\leq T}|e^{2Kt}X_t|^2\right]
%+C\mathbb{E}\left[\int_0^T e^{2Ks}(1+|X_s|^2+|u_s|^2)ds\right]<\infty.
%\end{aligned}
%\end{equation*}
Taking expectation on both sides of (\ref{X-Ito}) and using Assumptions (H1)-(H2)
together with the elementary inequality $ab\leq\varepsilon a^2+\frac{1}{\varepsilon}b^2$
for any $\varepsilon>0$, one has
\begin{equation*}
\begin{aligned}
&\mathbb{E}[|X_T e^{KT}|^2]\\
=&\mathbb{E}\bigg[|X_t e^{Kt}|^2
+\int_t^T e^{2Ks}[2K|X_s|^2+|\sigma(X_s,\alpha_s,u_s)-\sigma(0,\alpha_s,u_s)
+\sigma(0,\alpha_s,u_s)|^2]ds\bigg]\\
&+\mathbb{E}\bigg[\int_t^T e^{2Ks}[2\langle X_s,b(X_s,\alpha_s,u_s)-b(0,\alpha_s,u_s)\rangle
+2\langle X_s,b(0,\alpha_s,u_s)\rangle]ds\bigg]\\
\leq&\mathbb{E}\bigg[|X_t e^{Kt}|^2
+\int_t^T e^{2Ks}\bigg((2K+\kappa_{\sigma}^{2}+\varepsilon)|X_s|^2
+\frac{\kappa_{\sigma}^{2}}{\varepsilon}|\sigma(0,\alpha_s,u_s)|^2\bigg)ds\bigg]\\
&+\mathbb{E}\bigg[\int_t^T e^{2Ks}\bigg(-2\kappa_1|X_s|^2+\varepsilon|X_s|^2
+\frac{1}{\varepsilon}|b(0,\alpha_s,u_s)|^2\bigg)ds\bigg]\\
\leq&\mathbb{E}\bigg[|X_t e^{Kt}|^2
+(-2\kappa_1+2K+\kappa_{\sigma}^{2}+2\varepsilon)\int_t^T e^{2Ks}|X_s|^2ds\bigg]
+\frac{C}{\varepsilon}\mathbb{E}\bigg[\int_t^T e^{2Ks}(1+|u_s|^2)ds\bigg].
\end{aligned}
\end{equation*}
Therefore,
\begin{equation}\label{X-estimate}
\begin{aligned}
&\mathbb{E}[|X_T e^{KT}|^2]
+(2\kappa_1-2K-\kappa_{\sigma}^{2}-2\varepsilon)
\mathbb{E}\bigg[\int_t^T e^{2Ks}|X_s|^2ds\bigg]\\
\leq&\mathbb{E}[|X_t e^{Kt}|^2]
+\frac{C}{\varepsilon}\mathbb{E}\bigg[\int_t^T e^{2Ks}(1+|u_s|^2)ds\bigg].
\end{aligned}
\end{equation}
Now choosing $K<(\kappa_1-\frac{1}{2}\kappa_{\sigma}^2)\wedge0$
and $\varepsilon<\kappa_1-\frac{1}{2}\kappa_{\sigma}^2-K$, we get (let $t=0$)
\begin{equation*}
\begin{aligned}
\mathbb{E}\bigg[\int_0^{T} e^{2Ks}|X_s|^2ds\bigg]
\leq |x|^2+C\mathbb{E}\bigg[\int_0^\infty e^{2Ks}(1+|u_s|^2)ds\bigg].
\end{aligned}
\end{equation*}
Sending $T\rightarrow\infty$ and applying the monotonic convergence theorem,
we have $X\in L_{\mathbb{F}}^{2,K}(0,\infty;\mathbb{R}^n)$.

Next, we show (\ref{Xlim}) holds. Substituting $T$ with $T_2$ and $t$ with $T_1$
in (\ref{X-estimate}) leads to
\begin{equation}\label{Xcontin}
\begin{aligned}
&\Big|\mathbb{E}[|X_{T_1}e^{KT_1}|^2]-\mathbb{E}[|X_{T_2}e^{KT_2}|^2]\Big|\\
\leq&\frac{C}{\varepsilon}\mathbb{E}\bigg[\int_{T_1}^{T_2}e^{2Ks}(1+|u_s|^2)ds\bigg]
+(2\kappa_1-2K-\kappa_{\sigma}^{2}-2\varepsilon)\mathbb{E}\bigg[\int_{T_1}^{T_2}e^{2Ks}|X_s|^2ds\bigg].
\end{aligned}
\end{equation}
In view of $X\in L_{\mathbb{F}}^{2,K}(0,\infty;\mathbb{R}^n)$ and $u\in L_{\mathbb{F}}^{2,K}(0,\infty;\mathbb{R}^k)$,
it follows that the map $T\mapsto\mathbb{E}[|X_T e^{KT}|^2]$ is uniformly continuous.
Then, $\mathbb{E}[\int_0^{\infty}e^{2Ks}|X_s|^2ds]<\infty$ gives the asymptotic property (\ref{Xlim}).
\end{proof}

\subsection{Infinite horizon BSDEs with Markov chains}\label{subsection BSDEs}

First, the \emph{Hamiltonian} of our discounted optimal control problem is given by
\begin{equation}\label{Hamiltonian}
\begin{aligned}
H(x,i,u,y,z)=\langle b(x,i,u),y\rangle+\langle\sigma(x,i,u),z\rangle+f(x,i,u)-r\langle x,y\rangle.
\end{aligned}
\end{equation}
Then, the \emph{adjoint equation} reads
\begin{equation*}
\begin{aligned}
dY_{t}=-H_{x}(X_{t},\alpha_{t},u_{t},Y_{t},Z_{t})dt+Z_{t}dW_{t}+\Lambda_{t}\bullet dM_{t},\quad t\geq0,
\end{aligned}
\end{equation*}
i.e.,
\begin{equation}\label{adjoint}
\begin{aligned}
dY_{t}=&-[b_{x}(X_{t},\alpha_{t},u_{t})Y_{t}+\sigma_{x}(X_{t},\alpha_{t},u_{t})Z_{t}
+f_{x}(X_{t},\alpha_{t},u_{t})-rY_{t}]dt\\
&+Z_{t}dW_{t}+\Lambda_{t}\bullet dM_{t},\quad t\geq0.
\end{aligned}
\end{equation}
Here, $\phi_x$ denotes the partial derivative of $\phi=b,\sigma,f,H$ with respect to $x$.
This subsection is devoted to the global well-posedness of adjoint equation (\ref{adjoint}).
More generally, let's first consider the following general infinite horizon BSDE with Markov chain:
\begin{equation}\label{BSDE}
-dY_t=[F(Y_t,Z_t,\alpha_t)+\varphi_t]dt-Z_t dW_t-\Lambda_t\bullet dM_t,\quad t\geq0,
\end{equation}
where $F:\mathbb{R}^{n}\times\mathbb{R}^{n}\times\mathcal{M}\mapsto\mathbb{R}^{n}$ is a given function
and $\varphi$ is a given process in $L_{\mathbb{F}}^{2,K}(0,\infty;\mathbb{R}^{n})$. For convenience,
we denote $l_0=\mathbb{E}[\int_0^{\infty}e^{2Kt}|\varphi_t|^2dt]<\infty$.
\begin{definition}\label{def:YBSDE}
A triple $(Y,Z,\Lambda)\in L_{\mathbb{F}}^{2,K}(0,\infty;\mathbb{R}^n)
\times L_{\mathbb{F}}^{2,K}(0,\infty;\mathbb{R}^{n})\times M_{\mathbb{F}}^{2,K}(0,\infty;\mathbb{R}^{n})$
is called a solution to (\ref{BSDE}) if for any $T>0$, it satisfies
$$
Y_t=Y_T+\int_t^T[F(Y_s,Z_s,\alpha_s)+\varphi_s]ds
-\int_t^T Z_s dW_s-\int_t^T \Lambda_s\bullet dM_s,\quad t\in[0,T].
$$
\end{definition}
We introduce the following assumptions on the coefficients $(F,\varphi)$ of the BSDE (\ref{BSDE}).

(A1) The function $F$ is Lipschitz continuous with respect to $y,z$,
i.e., there exist constants $\kappa_{Fy}>0$ and $\kappa_{Fz}>0$ such that
$$
|F(y,z,i)-F(\overline{y},\overline{z},i)|\leq\kappa_{Fy}|y-\overline{y}|+\kappa_{Fz}|z-\overline{z}|.
$$
In addition, $F(0,0,i)=0$ for any $i\in\mathcal{M}$.

(A2) The function $F$ satisfies a monotonic condition with respect to $y$
in the sense that there exists a constant $\kappa_2\in\mathbb{R}$ such that
$$
\langle F(y,z,i)-F(\overline{y},z,i),y-\overline{y}\rangle\leq-\kappa_2|y-\overline{y}|^2.
$$
Next, we present a priori estimate for $Y$, the first component of the solution $(Y,Z,\Lambda)$
to BSDE (\ref{BSDE}), on any finite horizon $[0,T]$. Such an estimate will be used in
Lemma \ref{lem:priori}, which provides a priori estimate for $(Y,Z,\Lambda)$ on the infinite
horizon $[0,\infty)$.
\begin{lemma}\label{lem:Y}
Let Assumption (A1) hold. Let $(Y,Z,\Lambda)\in L_{\mathbb{F}}^{2,K}(0,\infty;\mathbb{R}^n)
\times L_{\mathbb{F}}^{2,K}(0,\infty;\mathbb{R}^{n})\times M_{\mathbb{F}}^{2,K}(0,\infty;\mathbb{R}^{n})$
be a solution to BSDE (\ref{BSDE}). Then, for any $T>0$, there exists a constant $C>0$,
which depends on $\kappa_{Fy}$, $\kappa_{Fz}$, and $K$, such that
\begin{equation}\label{Y-SK}
\begin{aligned}
\mathbb{E}\bigg[\sup_{0\leq t\leq T}|Y_t e^{Kt}|^2\bigg]
\leq&C\bigg(|Y_0|^2+l_0+\mathbb{E}\bigg[\int_0^T e^{2Ks}[|Y_s|^2+|Z_s|^2]ds\bigg]\\
&+\mathbb{E}\bigg[\int_0^Te^{2Ks}|\Lambda_s|^2\bullet d[M]_s\bigg]\bigg),
\end{aligned}
\end{equation}
and
\begin{equation}\label{Y-lim}
\lim_{T\rightarrow \infty}\mathbb{E}[|Y_T e^{KT}|^2]=0.
\end{equation}
\end{lemma}
\begin{proof}
The proof is postponed and given in the Appendix \ref{app:Y-lim}.
\end{proof}
%The following lemma gives a priori estimate for $(Y,Z,\Lambda)$
%in $L_{\mathbb{F}}^{2,K}(0,\infty;\mathbb{R}^n)\times L_{\mathbb{F}}^{2,K}(0,\infty;\mathbb{R}^{n})
%\times M_{\mathbb{F}}^{2,K}(0,\infty;\mathbb{R}^{n})$.
\begin{lemma}\label{lem:priori}
Let Assumptions (A1)-(A2) hold. Let $(Y,Z,\Lambda)\in L_{\mathbb{F}}^{2,K}(0,\infty;\mathbb{R}^n)
\times L_{\mathbb{F}}^{2,K}(0,\infty;\mathbb{R}^{n})\times M_{\mathbb{F}}^{2,K}(0,\infty;\mathbb{R}^{n})$
be a solution to BSDE (\ref{BSDE}). Then, for $K>-\kappa_2+\frac{\kappa_{Fz}^{2}}{2}$ and any $\varepsilon>0$,
we have
\begin{equation}\label{priori}
\begin{aligned}
&\mathbb{E}\bigg[|Y_0|^2+\int_0^{\infty}e^{2Ks}\bigg[(2K+2\kappa_2-\kappa_{Fz}^{2}-2\varepsilon)|Y_s|^2
+\frac{\varepsilon}{\kappa_{Fz}^{2}+\varepsilon}|Z_s|^2\bigg]ds\\
&+\int_0^{\infty}e^{2Ks}|\Lambda_s|^2\bullet d[M]_s\bigg]
\leq\frac{1}{\varepsilon} \mathbb{E}\bigg[\int_0^{\infty}e^{2Ks}|\varphi_s|^2 ds\bigg].
\end{aligned}
\end{equation}
In particular, there exists a constant $C>0$ such that
\begin{equation*}
\begin{aligned}
\mathbb{E}\bigg[|Y_0|^2+\int_0^{\infty}e^{2Ks}[|Y_s|^2+|Z_s|^2]ds
+\int_0^{\infty}e^{2Ks}|\Lambda_s|^2\bullet d[M]_s\bigg]\leq Cl_0.
\end{aligned}
\end{equation*}
\end{lemma}
\begin{proof}
Applying It\^{o}'s formula to $|Y_t e^{Kt}|^2$ on $[0,T]$ yields
\begin{equation}\label{ItoY}
\begin{aligned}
&|Y_T e^{KT}|^2\\
=&|Y_0|^2+\int_0^T e^{2Ks}[2K|Y_s|^2-2\langle Y_s,F(Y_s,Z_s,\alpha_s)+\varphi_s\rangle+|Z_s|^2]ds\\
&+\int_0^T e^{2Ks}|\Lambda_s|^2\bullet d[M]_s+\int_0^T e^{2Ks}\langle Y_s,Z_s\rangle dW_s
+\int_0^T e^{2Ks}\langle Y_s,\Lambda_s\bullet dM_s\rangle.
\end{aligned}
\end{equation}
Taking expectation on both sides of (\ref{ItoY}) leads to
\begin{equation}\label{expecta}
\begin{aligned}
&\mathbb{E}\bigg[|Y_0|^2+\int_0^T e^{2Ks}[2K|Y_s|^2+|Z_s|^2]ds
+\int_0^T e^{2Ks}|\Lambda_s|^2\bullet d[M]_s\bigg]\\
=&\mathbb{E}\bigg[|Y_T e^{KT}|^2
+2\int_0^T e^{2Ks}[\langle Y_s,F(Y_s,Z_s,\alpha_s)-F(0,Z_s,\alpha_s)\rangle]ds\bigg]\\
&+2\mathbb{E}\bigg[\int_0^T e^{2Ks}\langle Y_s,F(0,Z_s,\alpha_s)-F(0,0,\alpha_s)\rangle ds\bigg]
+2\mathbb{E}\bigg[\int_0^T e^{2Ks}\langle Y_s,\varphi_s\rangle ds\bigg].
\end{aligned}
\end{equation}
From the monotonic condition of $F$ in Assumption (A2), for any $\varepsilon>0$, one has
\begin{equation}\label{inequ}
\begin{cases}
\langle Y_s,F(Y_s,Z_s,\alpha_s)-F(0,Z_s,\alpha_s)\rangle\leq-\kappa_2|Y_s|^2,\\
2\langle Y_s,F(0,Z_s,\alpha_s)-F(0,0,\alpha_s)\rangle\leq2\kappa_{Fz}|Y_s||Z_s|
\leq(\kappa_{Fz}^{2}+\varepsilon)|Y_s|^2+\frac{\kappa_{Fz}^{2}}{\kappa_{Fz}^{2}+\varepsilon}|Z_s|^2,\\
2\langle Y_s,\varphi_s\rangle\leq\varepsilon|Y_s|^2+\frac{1}{\varepsilon}|\varphi_s|^2.
\end{cases}
\end{equation}
Substituting (\ref{inequ}) into (\ref{expecta}) gives
\begin{equation*}
\begin{aligned}
&\mathbb{E}\bigg[|Y_0|^2+\int_0^T e^{2Ks}[2K|Y_s|^2+|Z_s|^2]ds+\int_0^T e^{2Ks}|\Lambda_s|^2\bullet d[M]_s\bigg]\\
\leq&\mathbb{E}\bigg[|Y_T e^{KT}|^2+\int_t^T e^{2Ks}\bigg[-2\kappa_2|Y_s|^2+(\kappa_{Fz}^{2}+\varepsilon)|Y_s|^2
+\frac{\kappa_{Fz}^{2}}{\kappa_{Fz}^{2}+\varepsilon}|Z_s|^2\bigg]ds\bigg]\\
&+\mathbb{E}\bigg[\int_0^T e^{2Ks}\bigg(\varepsilon |Y_s|^2+\frac{1}{\varepsilon}|\varphi_s|^2\bigg)ds\bigg].
\end{aligned}
\end{equation*}
It follows that
\begin{equation*}
\begin{aligned}
&\mathbb{E}\bigg[|Y_0|^2+\int_0^T e^{2Ks}\bigg[(2K+2\kappa_2-\kappa_{Fz}^{2}-2\varepsilon)|Y_s|^2
+\frac{\varepsilon}{\kappa_{Fz}^{2}+\varepsilon}|Z_s|^2\bigg]ds\\
&+\int_0^T e^{2Ks}|\Lambda_s|^2\bullet d[M]_s\bigg]
\leq\mathbb{E}\bigg[|Y_T e^{KT}|^2+\frac{1}{\varepsilon}\int_0^T e^{2Ks}|\varphi_s|^2ds\bigg].
\end{aligned}
\end{equation*}
By letting $T\rightarrow\infty$ and noting the asymptotic property (\ref{Y-lim}) in Lemma \ref{lem:Y},
we have the desired priori estimate (\ref{priori}).
\end{proof}
Based on Lemma \ref{lem:priori} and the method in Peng and Shi \cite{peng2000infinite},
we obtain the following theorem.
\begin{theorem}\label{thm:Y-wellpose}
Let Assumptions (A1)-(A2) hold. Then, for $K>-\kappa_2+\frac{\kappa_{Fz}^{2}}{2}$,
the BSDE (\ref{BSDE}) admits a unique solution $(Y,Z,\Lambda)\in L_{\mathbb{F}}^{2,K}(0,\infty;\mathbb{R}^n)
\times L_{\mathbb{F}}^{2,K}(0,\infty;\mathbb{R}^{n})\times M_{\mathbb{F}}^{2,K}(0,\infty;\mathbb{R}^{n})$.
\end{theorem}
\begin{proof}
The uniqueness follows immediately from the priori estimate (\ref{priori}).
In the next, we show the existence of a solution to BSDE (\ref{BSDE}).
For $N=1,2,\ldots$, we denote
$$
F_N(Y_s,Z_s,\alpha_s)=F(Y_s,Z_s,\alpha_s)+\varphi_N(s),
$$
where $\varphi_N(s)=\varphi_s1_{[0,N]}(s)$. It can be seen that $\varphi_N(s)$
converges to $\varphi_s$ in $L_{\mathbb{F}}^{2,K}(0,\infty;\mathbb{R}^n)$.

Consider the following finite horizon BSDE with Markov chain:
\begin{equation}\label{finitetime}
\left\{
\begin{aligned}
dY_{t}^{(N)}=&-\Big[F\Big(Y_{t}^{(N)},Z_{t}^{(N)},\alpha_t\Big)+\varphi_N(t)\Big]dt
+Z_{t}^{(N)}dW_t+\Lambda_{t}^{(N)}\bullet dM_t,\quad t\in[0,N],\\
Y_{N}^{(N)}=&0.
\end{aligned}
\right.
\end{equation}
For any $N\geq1$, the finite horizon BSDE (\ref{finitetime}) admits a unique solution
$(Y_{t}^{(N)},Z_{t}^{(N)},\Lambda_{t}^{(N)})$; see Nguyen et al. \cite[Theorem 3.4]{nguyen2020}.
Then, we can construct the following processes (for $\Phi=Y,Z,\Lambda$):
$$
\widetilde{\Phi}_{s}^{(N)}=\Phi_{s}^{(N)}1_{[0,N]}(s)+01_{(N,\infty)}(s),
$$
which satisfy the following infinite horizon BSDE:
$$
dY_t=-[F(Y_t,Z_t,\alpha_t)+\varphi_N(t)]dt+Z_t dW_t+\Lambda_t\bullet dM_t.
$$
From Lemma \ref{lem:priori}, there exists a constant $C>0$ such that for any $N,L\in\mathbb{N}$,
\begin{equation*}
\begin{aligned}
&\mathbb{E}\bigg[\int_0^\infty e^{2Ks}\Big[\Big|\widetilde{Y}_{s}^{(N)}-\widetilde{Y}_{s}^{(L)}\Big|^2
+\Big|\widetilde{Z}_{s}^{(N)}-\widetilde{Z}_{s}^{(L)}\Big|^2\Big]ds
+\int_0^{\infty} e^{2Ks}\Big|\widetilde{\Lambda}^{(N)}_s-\widetilde{\Lambda}^{(L)}_s\Big|^2\bullet d[M]_s\bigg]\\
\leq&C\mathbb{E}\bigg[\int_0^{\infty} e^{2Ks}|\varphi_N(s)-\varphi_L(s)|^2ds\bigg]
\equiv C\mathbb{E}\bigg[\int_{N\wedge L}^{N\vee L} e^{2Ks}|\varphi_s|^2ds\bigg],
\end{aligned}
\end{equation*}
which implies that $\{(\widetilde{Y}^{(N)},\widetilde{Z}^{(N)},\widetilde{\Lambda}^{(N)})\}_{N=1}^\infty$
is a Cauchy sequence in the product Banach space $L_{\mathbb{F}}^{2,K}(0,\infty;\mathbb{R}^n)
\times L_{\mathbb{F}}^{2,K}(0,\infty;\mathbb{R}^{n})\times M_{\mathbb{F}}^{2,K}(0,\infty;\mathbb{R}^{n})$.
We then denote by $(Y,Z,\Lambda)$ the limit of the sequence, i.e.,
\begin{equation*}
\begin{aligned}
\lim_{N\rightarrow\infty}\mathbb{E}\bigg[\int_0^\infty e^{2Ks}\Big[\Big|\widetilde{Y}_{s}^{(N)}-Y_s\Big|^2
+\Big|\widetilde{Z}_{s}^{(N)}-Z_s\Big|^2\Big]ds
+\int_0^{\infty}e^{2Ks}\Big|\widetilde{\Lambda}^{(N)}_s-{\Lambda}_s\Big|^2\bullet d[M]_s\bigg]=0.
\end{aligned}
\end{equation*}
Note that, for any $T>0$ and any $N>T$,
\begin{equation*}
\begin{aligned}
\widetilde{Y}_{t}^{(N)}=\widetilde{Y}_{T}^{(N)}
+\int_t^T\Big[F\Big(\widetilde{Y}_{s}^{(N)},\widetilde{Z}_{s}^{(N)},\alpha_s\Big)+\varphi_N(s)\Big]ds&\\
-\int_t^T\widetilde{Z}_{s}^{(N)}dW_s-\int_t^T \widetilde{\Lambda}_{s}^{(N)}\bullet dM_s&,\quad t\in[0,T].
\end{aligned}
\end{equation*}
By sending $N\to\infty$, it follows that the limit $(Y,Z,\Lambda)$ turns out to be a solution to BSDE (\ref{BSDE}).
\end{proof}
Now, we return to the adjoint equation (\ref{adjoint}) of our discounted optimal control problem.
Clearly, this equation is a special case of BSDE (\ref{BSDE}). Hence, it is natural to impose
the following assumptions on the equation (\ref{adjoint}).

(H4) The functions $b$ and $\sigma$ are continuously differentiable with respect to $x$.
Furthermore, there exist constants $\kappa_{B}>0$, $\kappa_{\Sigma}>0$, and $\lambda_b\in\mathbb{R}$
such that
$$
B(x,i,u)\leq \kappa_B I_{n\times n},\quad
\Sigma(x,i,u)\leq\kappa_{\Sigma} I_{n\times n},\quad
\mathbb{S}(b_x(x,i,u))\leq \lambda_b I_{n\times n},
$$
where $I_{n\times n}$ denotes the $(n\times n)$ identity matrix,
$B(x,i,u)\doteq b_x^{\top}(x,i,u)b_x(x,i,u)$, $\Sigma(x,i,u)\doteq\sigma_x^{\top}(x,i,u)\sigma_x(x,i,u)$,
and $\mathbb{S}(A)\doteq\frac{1}{2}(A+A^{\top})$ represents the symmetrization of a matrix $A$.
In addition, for two symmetric matrices $M,N$, $M\leq N$ means $N-M$ is positive semi-definite.

(H5) The function $f$ is continuously differentiable with respect to $x$,
and there exists a constant $C>0$ such that
$$
f_x(x,i,u)\leq C(1+|x|+|u|).
$$

(H6) The discount factor $r>(\kappa_{\sigma}^{2}-2\kappa_1)\vee (\kappa_{\Sigma}+2\lambda_b)$.

From Assumption (H4), we have
\begin{equation*}
\begin{aligned}
|\sigma_x(x,i,u)(z-\overline{z})|
=\Big[\Big\langle\sigma_x(x,i,u)(z-\overline{z}),\sigma_x(x,i,u)(z-\overline{z})\Big\rangle\Big]^{\frac{1}{2}}
\leq\kappa_{\Sigma}^{\frac{1}{2}}|z-\overline{z}|,
\end{aligned}
\end{equation*}
and similarly,
\begin{equation*}
\begin{aligned}
|b_x(x,i,u)(y-\overline{y})|\leq\kappa_{B}^{\frac{1}{2}}|y-\overline{y}|,\quad
\Big\langle b_x(x,i,u)(y-\overline{y}),y-\overline{y}\Big\rangle \leq\lambda_b|y-\overline{y}|^2.
\end{aligned}
\end{equation*}
Moreover, for any $X\in L_{\mathbb{F}}^{2,K}(0,\infty;\mathbb{R}^n)$
and $u\in L_{\mathbb{F}}^{2,K}(0,\infty;\mathbb{R}^k)$, we have
\begin{equation*}
\begin{aligned}
\mathbb{E}\bigg[\int_0^{\infty}e^{2Ks}|f_x(X_s,\alpha_s,u_s)|^2ds\bigg]
\leq C\mathbb{E}\bigg[\int_0^{\infty}e^{2Ks}(1+|X_s|^2+|u_s|^2)ds\bigg]<\infty.
\end{aligned}
\end{equation*}
So the integrability requirement for the inhomogeneous term in (\ref{adjoint}) is also satisfied.
From Theorem \ref{thm:Y-wellpose}, we have the following result.
\begin{corollary}\label{cor:Y-wellpose}
Let Assumptions (H4)-(H6) hold and $\lambda_b+\frac{\kappa_{\Sigma}}{2}-r<K<0$. Then, for any given
$(X,u)\in L_{\mathbb{F}}^{2,K}(0,\infty;\mathbb{R}^n)\times L_{\mathbb{F}}^{2,K}(0,\infty;\mathbb{R}^k)$,
the adjoint equation (\ref{adjoint}) admits a unique solution
$(Y,Z,\Lambda)\in L_{\mathbb{F}}^{2,K}(0,\infty;\mathbb{R}^n)
\times L_{\mathbb{F}}^{2,K}(0,\infty;\mathbb{R}^{n})
\times M_{\mathbb{F}}^{2,K}(0,\infty;\mathbb{R}^{n})$.
\end{corollary}

\subsection{The associated stochastic Hamiltonian system}\label{subsection Hamiltonian system}

In our problem, the associated stochastic Hamiltonian system takes the following form
(a kind of infinite horizon FBSDE with Markov chain):
\begin{equation}\label{FBSDE}
\left\{
\begin{aligned}
dX_t=&b(X_t,\alpha_t,u_t)dt+\sigma(X_t,\alpha_t,u_t)dW_t,\\
-dY_t=&[b_x(X_t,\alpha_t,u_t)Y_t+\sigma_x(X_t,\alpha_t,u_t)Z_t+f_x(X_t,\alpha_t,u_t)-rY_t]dt\\
&-Z_t dW_t-\Lambda_t\bullet dM_t,\quad t\geq0,\\
X_0=&x\in\mathbb{R}^{n},\quad\alpha_0=i\in\mathcal{M}.
\end{aligned}
\right.
\end{equation}
In the following theorem, we will find a \emph{common} $K(=-\frac{r}{2})$ so that
both the SDE and the BSDE in (\ref{FBSDE}) admit unique solutions.
\begin{theorem}\label{thm:XY-wellpose}
Let Assumptions (H1)-(H6) hold. Then, for any $u\in\mathcal{U}_{ad}$,
the stochastic Hamiltonian system (\ref{FBSDE}) admits a unique solution
$(X,Y,Z,\Lambda)\in (L_{\mathbb{F}}^{2,-\frac{r}{2}}(0,\infty;\mathbb{R}^n))^{3}
\times M_{\mathbb{F}}^{2,-\frac{r}{2}}(0,\infty;\mathbb{R}^{n})$.
Moreover, we have
\begin{equation}\label{XYlimt}
\lim_{T\rightarrow\infty}\mathbb{E}[e^{-rT}|X_T|^2]=0,\quad
\lim_{T\rightarrow\infty}\mathbb{E}[e^{-rT}|Y_T|^2]=0.
\end{equation}
\end{theorem}
\begin{proof}
From Theorem \ref{thm:X-wellpose}, for any $u\in L_{\mathbb{F}}^{2,K}(0,\infty;\mathbb{R}^k)$,
the SDE in (\ref{FBSDE}) admits a unique solution $X\in L_{\mathbb{F}}^{2,K}(0,\infty;\mathbb{R}^n)$
for $K<(\kappa_1-\frac{\kappa_{\sigma}^2}{2})\wedge 0$. Given such an $X$, then
from Corollary \ref{cor:Y-wellpose}, the BSDE in (\ref{FBSDE}) admits a unique solution
$(Y,Z,\Lambda)\in L_{\mathbb{F}}^{2,K}(0,\infty;\mathbb{R}^n)\times L_{\mathbb{F}}^{2,K}(0,\infty;\mathbb{R}^{n})
\times M_{\mathbb{F}}^{2,K}(0,\infty;\mathbb{R}^{n})$ for $\lambda_b+\frac{\kappa_{\Sigma}}{2}-r<K<0$.
On the other hand, by noting the constraint
$r>(\kappa_{\sigma}^{2}-2\kappa_1)\vee (\kappa_{\Sigma}+2\lambda_b)$ in Assumption (H6), we have
$$
\lambda_b+\frac{\kappa_{\Sigma}}{2}-r<-\frac{r}{2}<\bigg(\kappa_1-\frac{\kappa_{\sigma}^{2}}{2}\bigg)\wedge0,
$$
which gives a room that we can choose $K=-\frac{r}{2}$ so that the FBSDE (\ref{FBSDE})
admits a unique solution $(X,Y,Z,\Lambda)\in (L_{\mathbb{F}}^{2,-\frac{r}{2}}(0,\infty;\mathbb{R}^n))^{3}
\times M_{\mathbb{F}}^{2,-\frac{r}{2}}(0,\infty;\mathbb{R}^{n})$.
Moreover, (\ref{Xlim}) in Theorem \ref{thm:X-wellpose} and (\ref{Y-lim}) in Lemma \ref{lem:Y}
yield the the asymptotic property (\ref{XYlimt}) holds.
\end{proof}

\section{Sufficient stochastic maximum principle}\label{section SMP}

In this section, we establish a sufficient SMP for the infinite horizon discounted
optimal control problem with regime switching considered in this paper. The key
tools used in the proof include a dual method, convex analysis, and the asymptotic
property given by (\ref{XYlimt}).
\begin{theorem}\label{thm:SMP}
Let Assumptions (H1)-(H6) hold. Let $u^{*}\in\mathcal{U}_{ad}$ be an admissible control
and $(X^*,Y^*,Z^*,\Lambda^*)$ be the corresponding unique solution to the stochastic Hamiltonian
system (\ref{FBSDE}). Moreover, suppose that:

{\rm(\romannumeral1)} $H(X_{t}^{*},\alpha_{t},u_{t}^{*},Y_{t}^{*},Z_{t}^{*})
=\min_{u\in \mathbb{R}^k}H(X_{t}^{*},\alpha_{t},u,Y_{t}^{*},Z_{t}^{*})$,

{\rm(\romannumeral2)} The map $(x,u)\mapsto H(x,\alpha_{t},u,Y_{t}^{*},Z_{t}^{*})$
is a convex function.

Then we have $u^{*}$ is an optimal control.
\end{theorem}
\begin{proof}
Taking an arbitrary $u\in\mathcal{U}_{ad}$. Let $X=X^{u}$ be the corresponding unique solution
to (\ref{FSDE}). By the definition of the Hamiltonian (\ref{Hamiltonian}), we have
\begin{equation}\label{vari}
\begin{aligned}
&J(x,i;u^*)-J(x,i;u)\\
=&\mathbb{E}\int_0^{\infty} e^{-rt}[f(X_{t}^{*},\alpha_{t},u_{t}^{*})-f(X_t,\alpha_{t},u_t)]dt\\
=&\mathbb{E}\int_0^{\infty} e^{-rt}\Big[H(X_{t}^{*},\alpha_t,u_{t}^{*},Y_{t}^{*},Z_{t}^{*})
-H(X_t,\alpha_t,u_t,Y_{t}^{*},Z_{t}^{*})+r\langle X_{t}^{*}-X_t,Y_{t}^{*}\rangle\\
&-\langle b(X_{t}^{*},\alpha_t,u_{t}^{*})-b(X_t,\alpha_t,u_t),Y_{t}^{*}\rangle
-\langle\sigma(X_{t}^{*},\alpha_t,u_{t}^{*})-\sigma(X_t,\alpha_t,u_t),Z_{t}^{*}\rangle\Big]dt.
\end{aligned}
\end{equation}
Applying It\^{o}'s formula to $e^{-rt}\langle X_{t}^{*}-X_t,Y_{t}^{*}\rangle$
on $[0,T]$ for any $T>0$, we have
\begin{equation}\label{erXY}
\begin{aligned}
&e^{-rT}\langle X_{T}^{*}-X_T,Y_{T}^{*}\rangle\\
=&\int_0^T e^{-rt}\Big[-r\Big\langle X_{t}^{*}-X_t,Y_{t}^{*}\Big\rangle
+\Big\langle X_{t}^{*}-X_t,H_x(X_{t}^{*},\alpha_t,u_{t}^{*},Y_{t}^{*},Z_{t}^{*})\Big\rangle\\
&+\Big\langle b(X_{t}^{*},\alpha_t,u_{t}^{*})-b(X_t,\alpha_t,u_t),Y_{t}^{*}\Big\rangle
+\Big\langle\sigma(X_{t}^{*},\alpha_t,u_{t}^{*})-\sigma(X_t,\alpha_t,u_t),Z_{t}^{*}\Big\rangle\Big]dt\\
&+\int_0^T e^{-rt}\Big[\Big\langle X_{t}^{*}-X_t,Z_{t}^{*}\Big\rangle
+\Big\langle\sigma(X_{t}^{*},\alpha_t,u_{t}^{*})-\sigma(X_t,\alpha_t,u_t),Y_{t}^{*}\Big\rangle\Big]dW_t\\
&+\int_0^T e^{-rt}\Big\langle X_{t}^{*}-X_t,\Lambda_{t}^{*}\bullet dM_t\Big\rangle.
\end{aligned}
\end{equation}
In view of (\ref{XYlimt}) in Theorem \ref{thm:XY-wellpose}, we obtain
\begin{equation}\label{XYlimt2}
\mathbb{E}[|e^{-rT}\langle X_{T}^{*}-X_T,Y_{T}^{*}\rangle|]
\leq\frac{1}{2}\mathbb{E}[e^{-rT}|X_{T}^{*}-X_T|^2]
+\frac{1}{2}\mathbb{E}[e^{-rT}|Y_{T}^{*}|^2]\rightarrow0,\quad T\rightarrow\infty.
\end{equation}
Taking expectation on both sides of (\ref{erXY}) and sending $T\rightarrow\infty$, we get
\begin{equation}\label{limit}
\begin{aligned}
0=&\lim_{T\rightarrow\infty}\mathbb{E}[e^{-rT}\langle X_{T}^{*}-X_T,Y_{T}^{*}\rangle]\\
=&\mathbb{E}\bigg[\int_0^{\infty} e^{-rt}\bigg(-r\Big\langle X_{t}^{*}-X_t,Y_{t}^{*}\Big\rangle
+\Big\langle X_{t}^{*}-X_t,H_x(X_{t}^{*},\alpha_t,u_{t}^{*},Y_{t}^{*},Z_{t}^{*})\Big\rangle\\
&+\Big\langle b(X_{t}^{*},\alpha_t,u_{t}^{*})-b(X_t,\alpha_t,u_t),Y_{t}^{*}\Big\rangle
+\Big\langle\sigma(X_{t}^{*},\alpha_t,u_{t}^{*})-\sigma(X_t,\alpha_t,u_t),Z_{t}^{*}\Big\rangle\bigg)dt\bigg].
\end{aligned}
\end{equation}
Plugging (\ref{limit}) into (\ref{vari}), one has
\begin{equation*}
\begin{aligned}
&J(x,i;u^*)-J(x,i;u)\\
=&\mathbb{E}\int_0^{\infty}e^{-rt}\bigg[H(X_{t}^{*},\alpha_t,u_{t}^{*},Y_{t}^{*},Z_{t}^{*})
-H(X_t,\alpha_t,u_t,Y_{t}^{*},Z_{t}^{*})\\
&-\Big\langle H_x(X_{t}^{*},\alpha_t,u_{t}^{*},Y_{t}^{*},Z_{t}^{*}),X_{t}^{*}-X_t\Big\rangle\bigg]dt.
\end{aligned}
\end{equation*}
By condition (\romannumeral1), we have $0\in\partial_u H(X_{t}^{*},\alpha_t,u_{t}^{*},Y_{t}^{*},Z_{t}^{*})$,
and together with condition (\romannumeral2), we have
$$
\Big(H_x(X_{t}^{*},\alpha_t,u_{t}^{*},Y_{t}^{*},Z_{t}^{*}),0\Big)
\in \partial_{x,u}H(X_{t}^{*},\alpha_t,u_{t}^{*},Y_{t}^{*},Z_{t}^{*}),
$$
where $\partial_u H$ (respectively, $\partial_{x,u} H$) denotes the \emph{Clarke generalized gradient}
of $H$ with respect to $u$ (respectively, $(x,u)$); see Yong and Zhou \cite[Chapter 3]{yong1999} for
more details about the generalized gradient. Then, it follows that
$$
H(X_{t}^{*},\alpha_t,u_{t}^{*},Y_{t}^{*},Z_{t}^{*})
-H(X_t,\alpha_t,u_t,Y_{t}^{*},Z_{t}^{*})
\leq\Big\langle H_x(X_{t}^{*},\alpha_t,u_{t}^{*},Y_{t}^{*},Z_{t}^{*}),X_{t}^{*}-X_t\Big\rangle.
$$
So we have
$$
J(x,i;u^*)-J(x,i;u)\leq0,
$$
i.e., $u^*$ is an optimal control.
\end{proof}
\begin{remark}
In comparison, the sufficient SMPs obtained by Donnelly \cite{donnelly2011}
and Haadem et al. \cite{haadem2013a} require an additional integrability condition
or transversality condition, respectively, which are not needed in Theorem \ref{thm:SMP}.

More precisely, Donnelly \cite[Theorem 4.1]{donnelly2011} imposed an integrability condition
to ensure that the stochastic integrals in (\ref{erXY}) are martingales. Such an integrability
condition can be naturally satisfied by the prior estimate in Lemma \ref{lem:priori}.
Haadem et al. \cite[Theorem 4.1]{haadem2013a} proposed a so-called transversality condition
for a reward maximization problem:
$$
0\leq\varlimsup_{t\to\infty}
\mathbb{E}\Big[\Big\langle X_t-X_t^*,Y_t^*\Big\rangle\Big]<\infty.
$$
Then, for a cost minimization problem, the above transversality condition becomes:
$$
-\infty<\varliminf_{t\to\infty}
\mathbb{E}\Big[\Big\langle X_t-X_t^{*},Y_t^{*}\Big\rangle\Big]\leq0,
$$
which can be verified by the asymptotic property of $X$ and $Y$ at infinity;
see (\ref{XYlimt}) and (\ref{XYlimt2}).
\end{remark}

\section{A production planning problem}\label{section production planning}

\subsection{Model and solution}\label{model-solution}

In this section, we apply the sufficient SMP (i.e., Theorem \ref{thm:SMP}) established in the last
section to solve a production planning problem in a regime switching market. An explicit optimal
production rate in a feedback form will be obtained via an algebraic Riccati equation (ARE for short)
and an additional linear algebraic equation. It is mentioned that the production planning problem
is a classical and important topic in the fields of management and manufacturing;
see \cite{thompson1980,bensoussan1984,fleming1987} and more recently \cite{GS2000ORL,GongZhou2013OR,WY2023OR}.
In this paper, the formulation of the production planning problem is modified by incorporating
the regime switching of the market into consideration, which, as a more realistic model, better
reflects the random market environment.

In this section, all variables and processes are assumed to be one-dimensional.
Let $X$ be the \emph{inventory level}, $u$ be the \emph{production rate}, and $\theta(i)$,
$i\in\mathcal{M}$, be the \emph{demand rates} under different market regimes.
The inventory process $X$ is described as
\begin{equation}\label{inventory-state}
\left\{
\begin{aligned}
dX_{t}=&(u_{t}-\theta(\alpha_{t}))dt+\sigma(\alpha_{t})dW_{t},\quad t\geq0,\\
X_{0}=&x,\quad \alpha_{0}=i,
\end{aligned}
\right.
\end{equation}
where $\sigma(i)$, $i\in \mathcal{M}$, are positive constants representing the \emph{volatility rates}
under different market regimes. The objective is to minimize the following cost functional:
$$
J(x,i;u)=\frac{1}{2}\mathbb{E}\int_{0}^{\infty}e^{-rt}
\Big[N(\alpha_{t})(X_{t}-c(\alpha_{t}))^{2}+R(\alpha_{t})(u_{t}-h(\alpha_{t}))^{2}\Big]dt,
$$
where the positive constants $c(i)$ and $h(i)$, $i\in \mathcal{M}$, are \emph{factory-optimal}
inventory levels and production rates under different market regimes, respectively, and
the positive constants $N(i)$ and $R(i)$, $i\in\mathcal{M}$, represent weighting coefficients
corresponding to inventory cost and production cost, respectively. Note that a negative
production rate $u_t$ can be interpreted as the scrapping rate of the existing inventory;
this situation occurs rarely when those $h(i)$, $i\in\mathcal{M}$, are large, but nevertheless
shall be permitted in the model (see \cite{thompson1980}).

The associated Hamiltonian of the problem is given by
\begin{equation*}
\begin{aligned}
H(x,i,u,y,z)
=(u-\theta (i))y+\sigma (i)z+\frac{1}{2}\Big[N(i)(x-c(i))^2+R(i)(u-h(i))^2\Big]-rxy.
\end{aligned}
\end{equation*}
It follows that
$$
\partial H_{x}=N(i)(x-c(i))-ry, \quad \partial H_{u}=y+R(i)(u-h(i)).
$$
Then, the corresponding adjoint equation reads
\begin{equation}\label{state-adjoint}
dY_{t}=-\Big[N(\alpha_{t})(X_{t}-c(\alpha_{t}))-rY_{t}\Big]dt
+Z_{t}dW_{t}+\Lambda_{t}\bullet dM_{t},\quad t\geq0,
\end{equation}
and, from Theorem \ref{thm:SMP}, the optimal control should have the following form:
\begin{equation}\label{op-1}
u_{t}=-\frac{Y_{t}}{R(\alpha_{t})}+h(\alpha_{t}).
\end{equation}
We conjecture that
\begin{equation}\label{conjecture}
\begin{aligned}
Y_{t}=\varphi(\alpha_{t})X_{t}+\psi(\alpha_{t}),
\end{aligned}
\end{equation}
with the two functions $\varphi:\mathcal{M}\to \mathbb{R}$
and $\psi:\mathcal{M}\to \mathbb{R}$ to be determined later.
Then, the optimal control (\ref{op-1}) becomes the following feedback form:
\begin{equation}\label{op-2}
u_{t}=-\frac{1}{R(\alpha_t)}[\varphi(\alpha_{t})X_{t}+\psi(\alpha_{t})]+h(\alpha_{t}).
\end{equation}
On the one hand, applying It\^{o}'s formula to (\ref{conjecture}), we have
\begin{equation*}
\begin{aligned}
dY_{t}=&\bigg[\varphi(\alpha_{t})(u_{t}-\theta(\alpha_{t}))
+\sum_{j\in\mathcal{M}}q_{\alpha_{t},j}[\varphi(j)X_{t}+\psi(j)]\bigg]dt\\
&+\varphi(\alpha_{t})\sigma(\alpha_{t})dW_{t}
+\sum_{i,j\in\mathcal{M}}\Big([\varphi(j)X_{t}+\psi(j)]-[\varphi(i)X_{t}+\psi(i)]\Big)dM_{ij}(t).
\end{aligned}
\end{equation*}
By inserting (\ref{op-2}) into the above equation, we get
\begin{equation}\label{compare-1}
\begin{aligned}
dY_{t}=&\bigg[\varphi(\alpha_{t})\bigg(-\frac{1}{R(\alpha_t)}[\varphi(\alpha_{t})X_{t}+\psi(\alpha_{t})]
+h(\alpha_{t})-\theta(\alpha_{t})\bigg)
+\sum_{j\in\mathcal{M}}q_{\alpha_{t},j}[\varphi(j)X_{t}+\psi(j)]\bigg]dt\\
&+\varphi(\alpha_{t})\sigma(\alpha_{t})dW_{t}
+\sum_{i,j\in\mathcal{M}}\Big([\varphi(j)X_{t}+\psi(j)]-[\varphi(i)X_{t}+\psi(i)]\Big)dM_{ij}(t).
\end{aligned}
\end{equation}
On the other hand, it follows from the original adjoint equation (\ref{state-adjoint}) that
\begin{equation}\label{compare-2}
\begin{aligned}
dY_{t}=&-[N(\alpha_{t})(X_{t}-c(\alpha_{t}))-rY_{t}]dt
+Z_{t}dW_{t}+\Lambda_{t}\bullet dM_{t}\\
=&-[N(\alpha_{t})(X_{t}-c(\alpha_{t}))-r(\varphi(\alpha_{t})X_{t}+\psi(\alpha_{t}))]dt
+Z_{t}dW_{t}+\Lambda_{t}\bullet dM_{t}.
\end{aligned}
\end{equation}
By comparing (\ref{compare-1}) and (\ref{compare-2}), we obtain the following ARE:
\begin{equation}\label{RE}
\frac{\varphi^2(i)}{R(i)}+r\varphi(i)-\sum_{j\in\mathcal{M}}{q_{ij}}\varphi(j)-N(i)=0,\quad i\in\mathcal{M},
\end{equation}
and a linear algebraic equation:
\begin{equation}\label{RE-2}
\bigg(r+\frac{\varphi(i)}{R(i)}\bigg)\psi(i)-\sum_{j\in\mathcal{M}}{q_{ij}}\psi(j)
-(h(i)-\theta(i))\varphi(i)+N(i)c(i)=0,\quad i\in\mathcal{M}.
\end{equation}
Note that the solution to the ARE (\ref{RE}) in general may not be unique. However, the following
lemma shows that a \emph{non-negative} solution to (\ref{RE}) is unique, which not only is
an inherent requirement from the production planning problem (see Remark \ref{rmk:optimal-state}),
but also further guarantees the uniqueness of solution to the linear algebraic equation (\ref{RE-2}).
\begin{lemma}\label{lem:Riccatisolvable}
The algebraic equations (\ref{RE}) and (\ref{RE-2}) admit unique solutions
$\varphi(i)$ and $\psi(i)$, respectively, with $\varphi(i)\geq 0$, $i\in\mathcal{M}$.
\end{lemma}
\begin{proof}
The proof is postponed and given in the Appendix \ref{app:Riccatisolvable}.
\end{proof}
\begin{theorem}\label{thm:optimal-u}
The optimal production rate $u^*$ for the production planning problem is given by
\begin{equation}\label{optimal-u}
u^*_{t}=-\frac{1}{R(\alpha_t)}[\varphi(\alpha_{t})X^*_{t}+\psi(\alpha_{t})]+h(\alpha_{t}),
\end{equation}
where $\varphi$ and $\psi$ are solutions to (\ref{RE}) and (\ref{RE-2}), respectively,
with $\varphi(i)\geq 0$, $i\in\mathcal{M}$, and $X^*$ is the optimal inventory process
corresponding to $u^*$. Moreover, the value function is given by
\begin{equation}\label{eq:value}
\begin{aligned}
v(x,i)=&\frac{1}{2}\varphi(i)x^2+\psi(i)x\\
&+\frac{1}{2}\mathbb{E}\int_0^{\infty} e^{-rt}\bigg[N(\alpha_{t})c^2(\alpha_t)
+\varphi(\alpha_t)\sigma^2(\alpha_t)-\frac{\psi^2(\alpha_t)}{R(\alpha_t)}
+2\psi(\alpha_t)(h(\alpha_t)-\theta(\alpha_t))\bigg]dt.
\end{aligned}
\end{equation}
\end{theorem}
\begin{proof}
We adopt the \emph{method of completing the square} to verify the optimality of the feedback
control (\ref{optimal-u}) and compute the value function (\ref{eq:value}) of the problem.
To this end, applying It\^{o}'s formula to $e^{-rt}\varphi(\alpha_{t})X_{t}^{2}$, we have
(for convenience, only the $dt$ part is preserved)
\begin{equation*}
\begin{aligned}
&d\bigg(\frac{1}{2}e^{-rt}\varphi(\alpha_{t})X_{t}^{2}\bigg)\\
=&\frac{1}{2}e^{-rt}\bigg[-r\varphi(\alpha_{t})X_{t}^{2}
+2\varphi(\alpha_{t})X_{t}(u_{t}-\theta(\alpha_{t}))
+\sum_{j\in\mathcal{M}}q_{\alpha_{t},j}\varphi(j)X_{t}^{2}
+\varphi(\alpha_{t})\sigma^{2}(\alpha_{t})\bigg]dt.
\end{aligned}
\end{equation*}
Then, applying It\^{o}'s formula to $e^{-rt}\psi(\alpha_{t})X_{t}$, one obtains
\begin{equation*}
\begin{aligned}
d[e^{-rt}\psi(\alpha_{t})X_{t}]
=e^{-rt}\bigg[-r\psi(\alpha_{t})X_{t}+\psi(\alpha_{t})(u_{t}-\theta(\alpha_{t}))
+\sum_{j\in\mathcal{M}}q_{\alpha_{t},j}\psi(j)X_{t}\bigg]dt.
\end{aligned}
\end{equation*}
Noting that
$$
\mathbb{E}[e^{-rT}\varphi(\alpha_T)X_{T}^{2}]\rightarrow0,\quad
\mathbb{E}[|e^{-rT}\psi(\alpha_T)X_T|]
=\mathbb{E}\left[|\psi(\alpha_T)|e^{-\frac{1}{2}rT}\sqrt{e^{-rT}X_{T}^{2}}\right]\rightarrow0,
$$
as $T\rightarrow\infty$ (see (\ref{XYlimt}) in Theorem \ref{thm:XY-wellpose}), we have
\begin{equation*}
\begin{aligned}
&J(x,i;u)-\frac{1}{2}\varphi(i)x^2-\psi(i)x\\
=&J(x,i;u)
+\mathbb{E}\int_0^{\infty}d\bigg(\frac{1}{2}e^{-rt}\varphi(\alpha_t)X_{t}^{2}\bigg)
+\mathbb{E}\int_0^{\infty}d[e^{-rt}\psi(\alpha_t)X_t]\\
=&\frac{1}{2}\mathbb{E}\int_0^{\infty}
e^{-rt}R(\alpha_t)\bigg|u_t+\frac{\varphi(\alpha_t)}{R(\alpha_t)}X_t
+\frac{\psi(\alpha_t)}{R(\alpha_t)}-h(\alpha_t)\bigg|^2dt\\
&+\frac{1}{2}\mathbb{E}\int_0^{\infty} e^{-rt}\bigg[N(\alpha_{t})c^2(\alpha_t)
+\varphi(\alpha_t)\sigma^2(\alpha_t)-\frac{\psi^2(\alpha_t)}{R(\alpha_t)}
+2\psi(\alpha_t)(h(\alpha_t)-\theta(\alpha_t))\bigg]dt.
\end{aligned}
\end{equation*}
It follows that $u^*$ defined by (\ref{optimal-u}) is an optimal control
and \eqref{eq:value} is the value function.
\end{proof}
\begin{remark}\label{rmk:optimal-state}
The requirement that $\varphi$ should be a non-negative solution arises from two aspects.
Firstly, the cost functional of the production planning problem is non-negative.
Hence, the value function (\ref{eq:value}) must be non-negative as well for any $(x,i)$.
Note that $v(x,i)$ is a quadratic function of $x$ with the coefficient of $x^{2}$
being $\varphi(i)$, which necessarily requires that $\varphi(i)\geq 0$, $i\in\mathcal{M}$.
Secondly, given a non-negative solution $\varphi$, the associated stochastic Hamiltonian
system with the optimal control $u^*$ given by (\ref{optimal-u}) is uniquely solvable for any $r>0$.
In fact, under (\ref{optimal-u}), the state equation (\ref{inventory-state}) becomes
\begin{equation}\label{optimal-state}
\left\{
\begin{aligned}
dX^*_t=&\Big(-\varphi(\alpha_t)X^*_t-\psi(\alpha_t)+h(\alpha_t)-\theta(\alpha_t)\Big)dt
+\sigma(\alpha_t)dW_t,\quad t\geq 0,\\
X^*_0=&x,\quad \alpha_0=i.
\end{aligned}
\right.
\end{equation}
So we can choose $\kappa_\sigma=0$ and $\kappa_1=\min_{i\in\mathcal{M}}\{\varphi(i)\}\geq0$
in Assumptions (H1)-(H3), by noticing that
$$
\langle b(x,i)-b(\overline{x},i),x-\overline{x}\rangle
=-\varphi(i)|x-\overline{x}|^2\leq-\kappa_1|x-\overline{x}|^2,
$$
and
$$
|\sigma(x,i)-\sigma(\overline{x},i)|^2=0\leq\kappa_{\sigma}|x-\overline{x}|^2.
$$
From Theorem \ref{thm:X-wellpose}, the state equation (\ref{optimal-state})
admits a unique solution $X^*\in L_{\mathbb{F}}^{2,-\frac{r}{2}}(0,\infty;\mathbb{R}^n)$
if $-\frac{r}{2}<(\kappa_1-\frac{1}{2}\kappa_{\sigma}^{2})\wedge0
=\min_{i\in\mathcal{M}}\{\varphi(i)\}\wedge0=0$, which is equivalent to $r>0$.
Then, by selecting $\kappa_{B}=\kappa_{\Sigma}=\lambda_b=0$ in Assumptions (H4)-(H6)
and from Corollary \ref{cor:Y-wellpose}, the following triple of processes defined by
$$
Y_{t}^{*}\doteq\varphi(\alpha_{t})X_{t}^{*}+\psi(\alpha_{t})
\in L_{\mathbb{F}}^{2,-\frac{r}{2}}(0,\infty;\mathbb{R}^n),\quad
Z_{t}^{*}\doteq\varphi(\alpha_t)\sigma(\alpha_t)\in L_{\mathbb{F}}^{2,-\frac{r}{2}}(0,\infty;\mathbb{R}^n),
$$
and
$$
\Lambda_{ij}^{*}(t)=[\varphi(j)-\varphi(i)]X_t^*+[\psi(j)-\psi(i)]
\in M_{\mathbb{F}}^{2,-\frac{r}{2}}(0,\infty;\mathbb{R}^n),
$$
turns out to be the unique solution to the adjoint equation (\ref{state-adjoint}).
\end{remark}

\subsection{Numerical experiments}\label{section numerics}

In the following, we provide a numerical example to illustrate the theoretical results
obtained in the last subsection. We will first consider a \emph{basic case} with a set
of benchmark model parameters, and then conduct a \emph{sensitivity analysis} by varying
some of these parameters.

\subsubsection{A basic case}\label{basic}

Let us consider a Markov chain $\alpha$ having two states $\{1,2\}$, which means
the market switches between two regimes 1 (``good" or ``bull") and 2 (``bad" or ``bear"),
and the generator of the Markov chain is supposed to be
$$
Q=
\left[
\begin{matrix}
-1& 1\\
1& -1\\
\end{matrix}
\right].
$$
Moreover, let $r=0.05$, and $\theta(i)$, $\sigma(i)$, $c(i)$, $h(i)$, $i\in\{1,2\}$, be given by
$$
\left[
\begin{array}{c}
	\theta (1)\\
	\theta (2)\\
\end{array}
\right]
=\left[
\begin{array}{c}
	4\\
	2.5\\
\end{array}
\right],
\quad
\left[
\begin{array}{c}
	\sigma (1)\\
	\sigma (2)\\
\end{array}
\right]
=\left[
\begin{array}{c}
	0.6\\
	0.8\\
\end{array}
\right],
\quad
\left[
\begin{array}{c}
	c(1)\\
	c(2)\\
\end{array}
\right]
=\left[
\begin{array}{c}
	3\\
	1.5\\
\end{array}
\right],
\quad
\left[
\begin{array}{c}
	h(1)\\
	h(2)\\
\end{array}
\right]
=\left[
\begin{array}{c}
	5\\
	4\\
\end{array}
\right],
$$
and the weighting coefficients of inventory cost and production cost are given by
$$
\left[
\begin{array}{c}
	N(1)\\
	N(2)\\
\end{array}
\right]
=\left[
\begin{array}{c}
	0.4\\
	0.3\\
\end{array}
\right],
\quad
\left[
\begin{array}{c}
	R(1)\\
	R(2)\\
\end{array}
\right]
=\left[
\begin{array}{c}
	0.5\\
	0.4\\
\end{array}
\right].
$$
Given the above set of benchmark model parameters, the solutions of the ARE (\ref{RE}) and
the linear algebraic equation (\ref{RE-2}) are computed to be:
$$
\left[
\begin{array}{c}
	\varphi(1)\\
	\varphi(2)\\
\end{array}
\right]
=\left[
\begin{array}{c}
	0.408\\
	0.362\\
\end{array}
\right],
\quad
\left[
\begin{array}{c}
	\psi(1)\\
	\psi(2)\\
\end{array}
\right]
=\left[
\begin{array}{c}
	-0.549\\
	-0.233\\
\end{array}
\right].
$$
Then, from Theorem \ref{thm:optimal-u}, the optimal production rate $u^{*}$ takes the following
feedback form of the inventory level $X^{*}$ and the market regime $i\in\{1,2\}$:
$$
u_{t}^{*}
=\begin{cases}
-0.816X_t^*+1.098+5,\quad \alpha_t=1,\\
-0.905X_t^*+0.582+4,\quad \alpha_t=2,
\end{cases}
$$
and the value function $v(x,i)$ and a simulation result of $X^*$ and $u^*$ are presented
in Figures \ref{Fig.sub.1} and \ref{Fig.sub.2}, respectively.

\begin{figure}[htbp]
\centering
\includegraphics[width=4in]{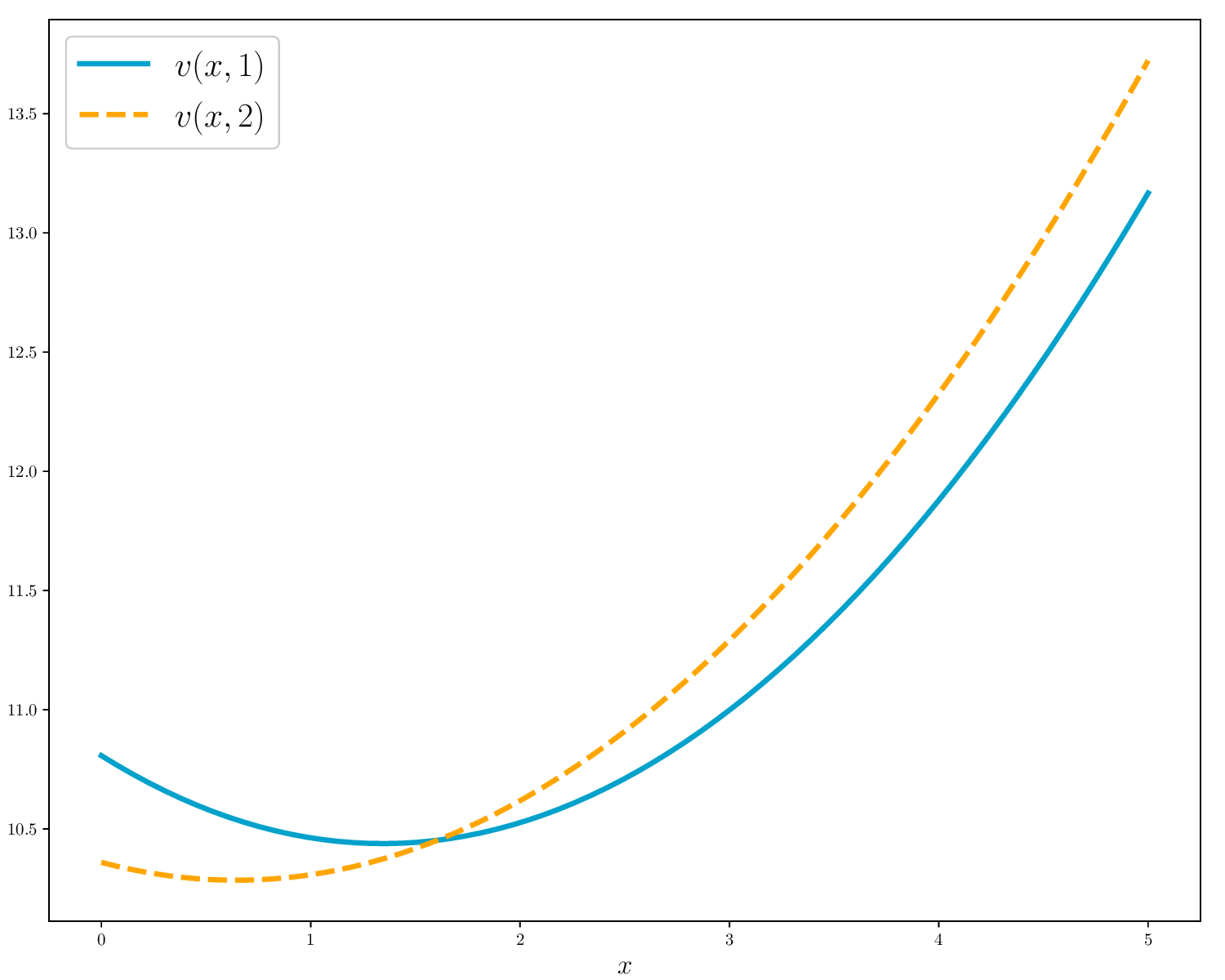}
\caption{Value function $v(x,i)$.}
\label{Fig.sub.1}
\end{figure}

\begin{figure}[htbp]
\centering
\includegraphics[width=4in]{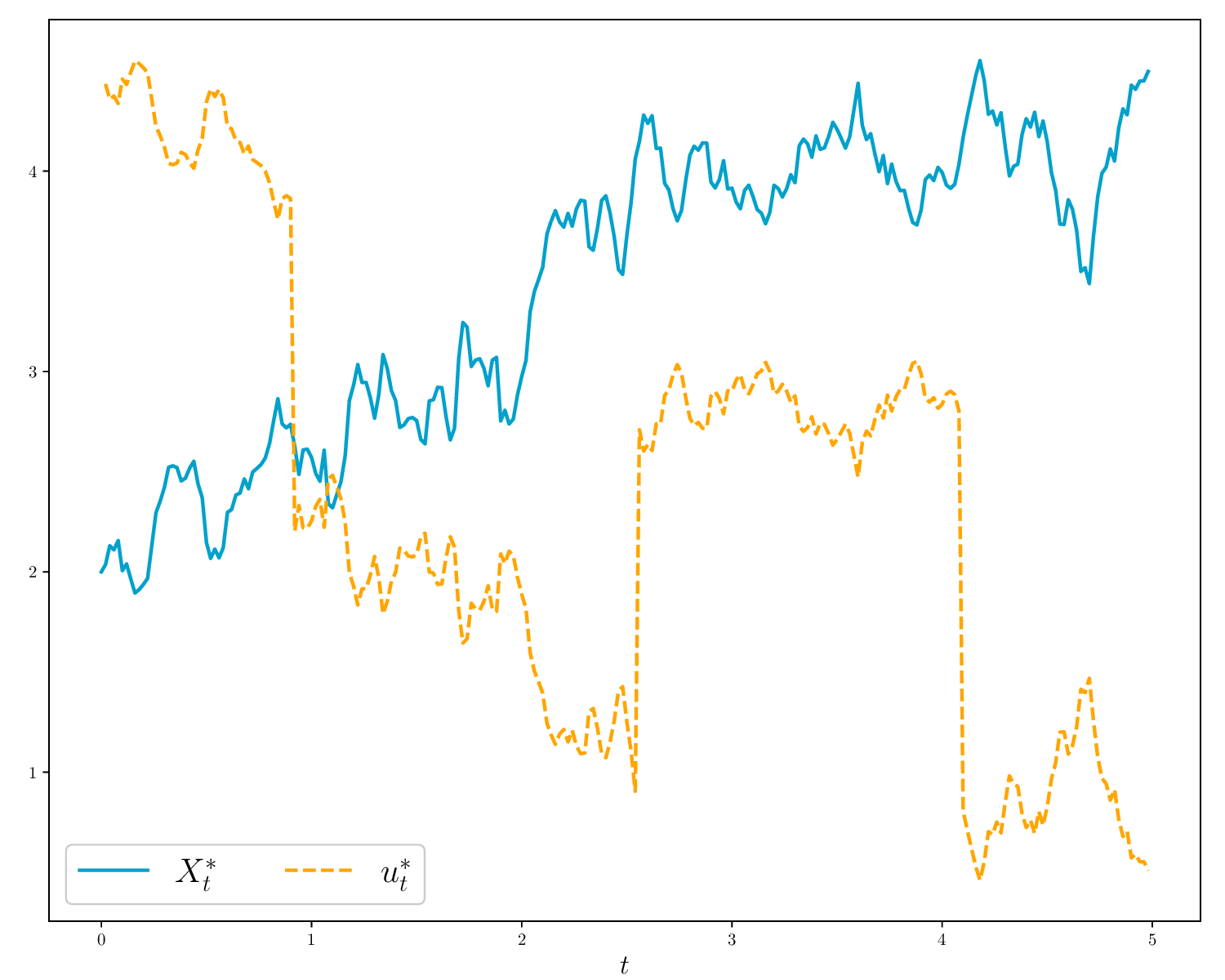}
\caption{Simulation of $X_t^*$ and $u_t^*$.}
\label{Fig.sub.2}
\end{figure}

\subsubsection{Sensitivity analysis}

In the next, based on the set of benchmark model parameters given in Subsubsction \ref{basic},
we present a sensitivity analysis (or monotonic dependence) of the value function $v(x,i)$
on four parameters of particular interests, including the discount factor $r$,
the generator $Q$, and the coefficients $\theta(i)$ and $\sigma(i)$ of the inventory process.
In the numerical experiments, we will vary only one of the four parameters at a time while
keeping the other three unchanged. On the other hand, the values of $N(i)$, $R(i)$, $c(i)$, $h(i)$
are fixed to be the same with those in Subsubsction \ref{basic}.

We first investigate the impact of the discount factor $r$ on $v(x,i)$ across three levels:
(\romannumeral1) a low level $r=0.03$, (\romannumeral2) a medium level $r=0.05$,
and (\romannumeral3) a high level $r=0.08$. Figure \ref{Fig.sub.3} shows that the value function
shifts downward as $r$ increases. Furthermore, Table \ref{table:phi_psi} indicates that
$\varphi$ falls as $r$ rises, which suggests that the negative feedback of the production rate
$u^*$ on the inventory level $X^*$ gets weakened.

%\begin{figure}[htbp]
%\centering
%\includegraphics[width=3.75in]{value_for_different_r.eps}
%\caption{Value functions for different $r$.}
%\label{Fig.sub.3}
%\end{figure}

Next, we assume $q_{12}=q_{21}\equiv q$ in $Q$ and analyze three different transition rates:
(\romannumeral1) $q=1$, (\romannumeral2) $q=2$, and (\romannumeral3) $q=5$. It can be seen
from Figure \ref{Fig.sub.4} that an increase in $q$ leads to an increase in $v(x,1)$ (good market)
and a decrease in $v(x,2)$ (bad market). This convergence (i.e., moving closer to each other)
of $v(x,1)$ and $v(x,2)$ reflects that as the two states switch more and more rapidly, they are
becoming more like a single state. The opposite monotonicity of $\varphi(1)$ vs. $\varphi(2)$
and $\psi(1)$ vs. $\psi(2)$ with respect to $q$ in Table \ref{table:phi_psi} supports
such a fact.

%\begin{figure}[htbp]
%\centering
%\includegraphics[width=3.75in]{value_for_different_q.eps}
%\caption{Value functions for different $q$.}
%\label{Fig.sub.4}
%\end{figure}

Then, the value functions for three different demand rates are plotted in Figure \ref{Fig.sub.5},
include: (\romannumeral1) $[\theta(1),\theta(2)]=[4,1.5]$, (\romannumeral2) $[\theta(1),\theta(2)]=[4,2.5]$,
and (\romannumeral3) $[\theta(1),\theta(2)]=[5,2.5]$. Increasing either $\theta(1)$ or $\theta(2)$
decreases the value function. This is natural as an increasing demand rate leads to a higher sale
opportunity, which further reduces the inventory cost. On the other hand, Table \ref{table:phi_psi}
shows that $\varphi$ remains the same while $\psi$ declines when $\theta$ grows up, which implies
a stability in the feedback component and an increase in the fixed component of the optimal production
rate $u_t^*$.

%\begin{figure}[htbp]
%\centering
%\includegraphics[width=3.75in]{value_for_different_theta.eps}
%\caption{Value functions for different $\theta$.}
%\label{Fig.sub.5}
%\end{figure}

Finally, we examine the effect of volatility rate $\sigma$ on the value function
with three levels: (\romannumeral1) $[\sigma(1),\sigma(2)]=[0.1,0.3]$, (\romannumeral2)
$[\sigma(1),\sigma(2)]=[0.4,0.6]$, and (\romannumeral3) $[\sigma(1),\sigma(2)]=[0.8,1.2]$.
We observe from Figure \ref{Fig.sub.6} that the value function increases as $\sigma$
gets larger. A larger volatility rate means a greater uncertainty, which results in
a bigger cost. From Table \ref{table:phi_psi}, $\varphi$ and $\psi$ remain unchanged
for different volatility rates, and so does the optimal production rate $u_t^*$.
The reason may be that the control $u$ does not enter the diffusion term of
the inventory process $X$.

%\begin{figure}[htbp]
%\centering
%\includegraphics[width=3.75in]{value_for_different_sigma.eps}
%\caption{Value functions for different $\sigma$.}
%\label{Fig.sub.6}
%\end{figure}

\begin{figure}[htbp]
\centering
\subfloat[Value functions for different $r$.]
{\includegraphics[width=0.45\textwidth]{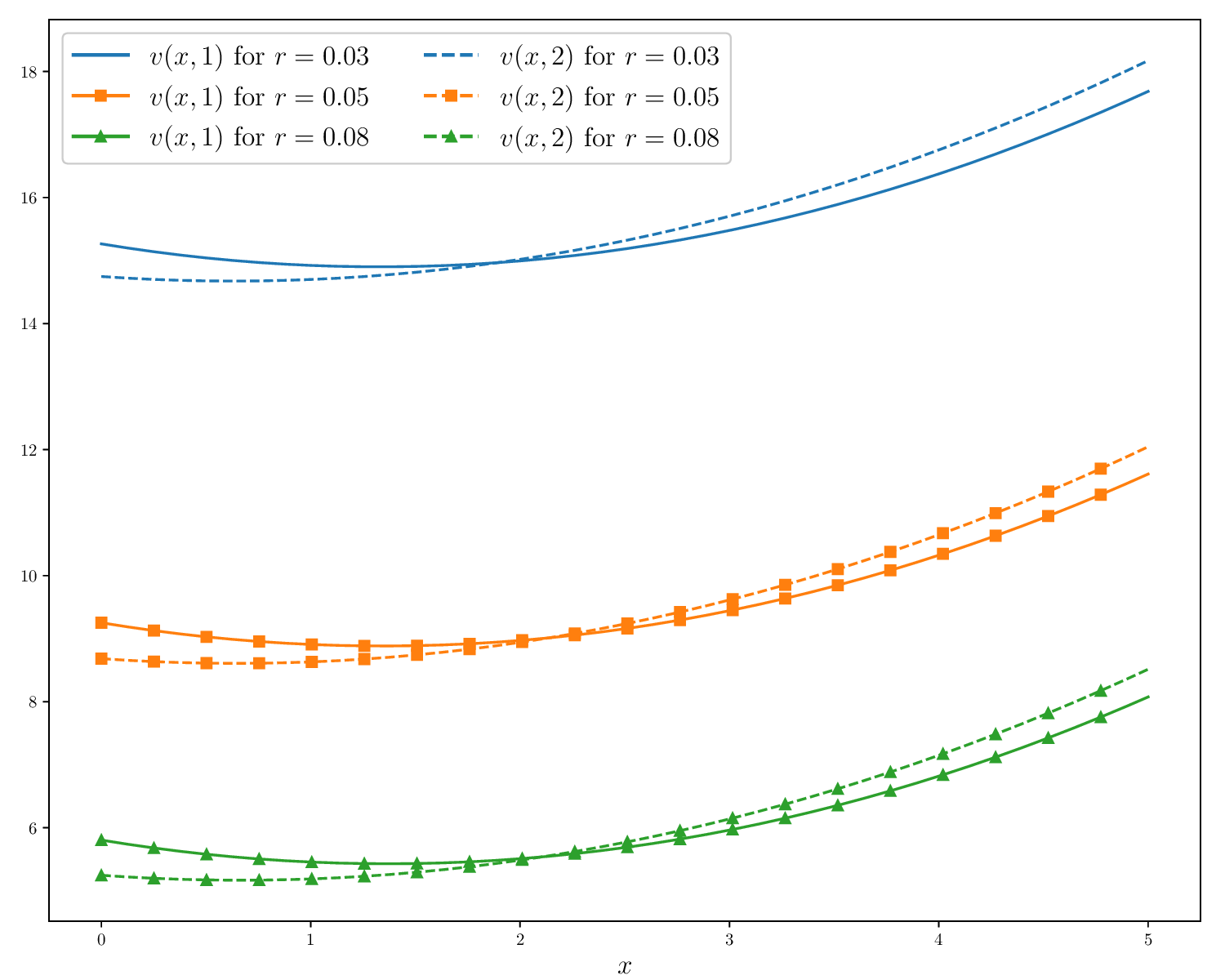}\label{Fig.sub.3}}
\subfloat[Value functions for different $q$.]
{\includegraphics[width=0.45\textwidth]{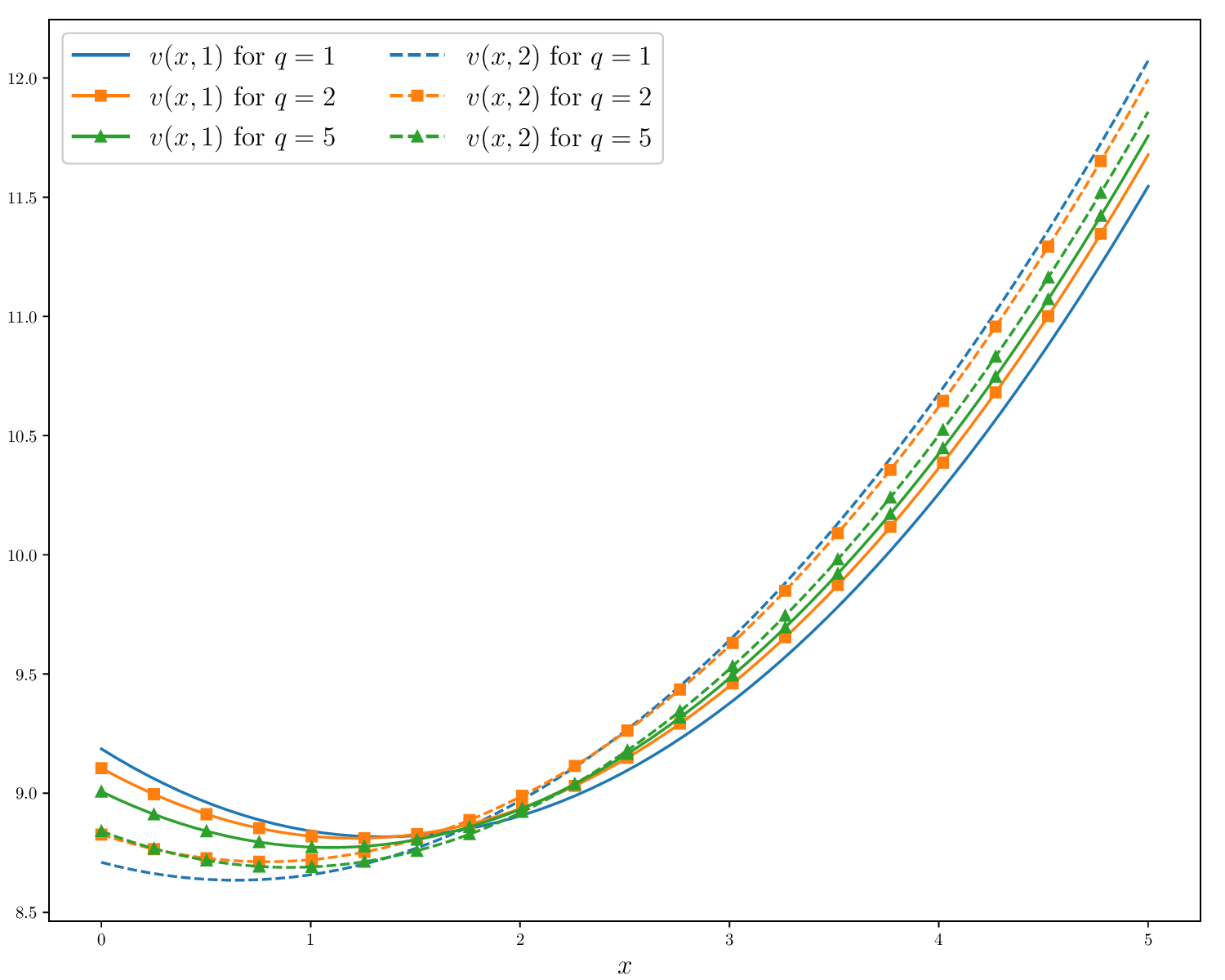}\label{Fig.sub.4}}\\
\subfloat[Value functions for different $\theta$.]
{\includegraphics[width=0.45\textwidth]{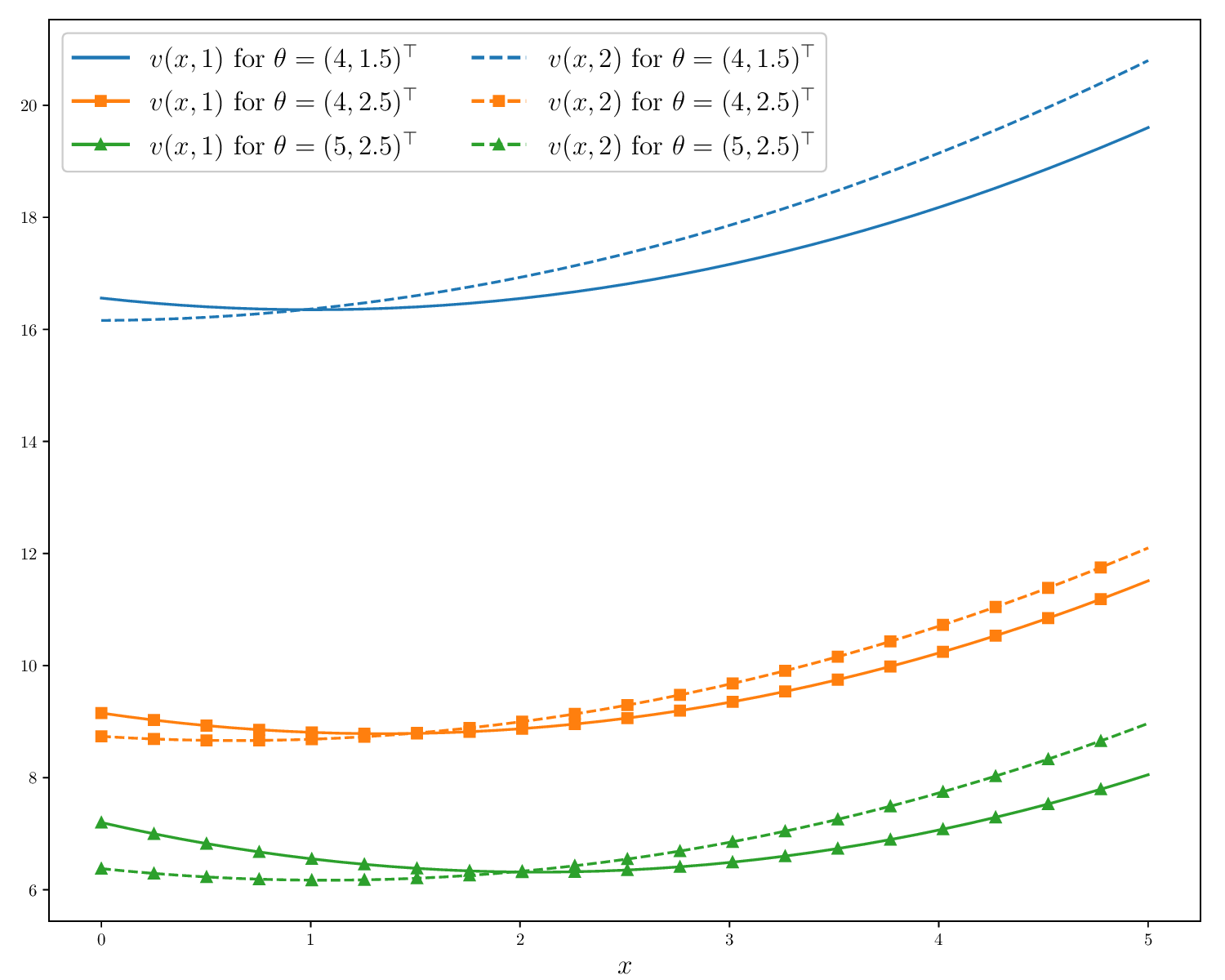}\label{Fig.sub.5}}
\subfloat[Value functions for different $\sigma$.]
{\includegraphics[width=0.45\textwidth]{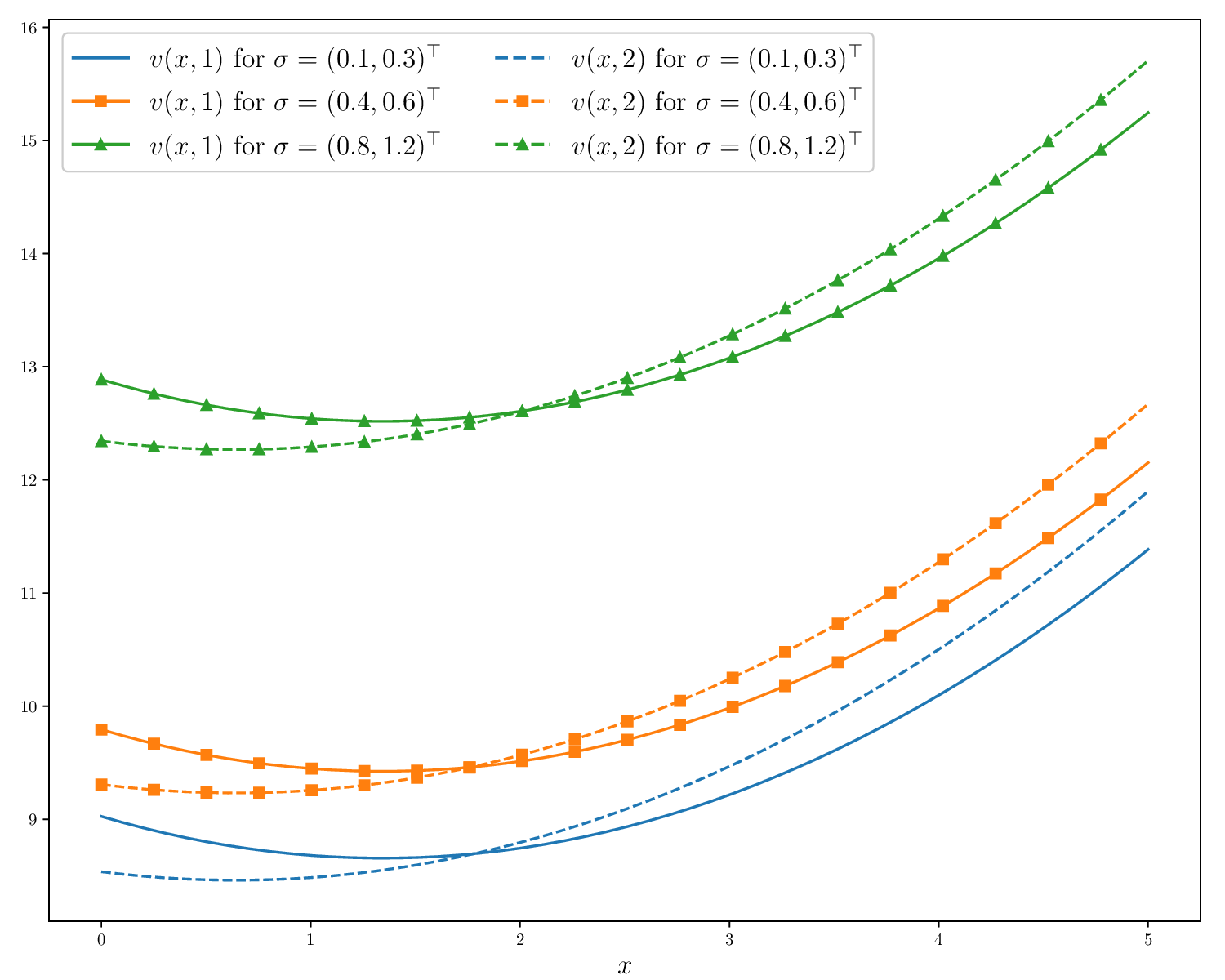}\label{Fig.sub.6}}
\caption{Dependency of value functions on model parameters.}
\end{figure}

\begin{table}[htbp]
\caption{Values of $\varphi$ and $\psi$ for different $r$, $q$, $\theta$, $\sigma$.}
\label{table:phi_psi}
\centering
\begin{tabularx}{0.9\textwidth}{>{\centering\arraybackslash}X>{\centering\arraybackslash}X
>{\centering\arraybackslash}X>{\centering\arraybackslash}X>{\centering\arraybackslash}X}
\toprule
& $\varphi(1)$ & $\varphi(2)$ & $\psi(1)$ & $\psi(2)$ \\
\midrule
$r=0.03$ & 0.413 & 0.366 & -0.548 & -0.230 \\
$r=0.05$ & 0.408 & 0.362 & -0.549 & -0.233 \\
$r=0.08$ & 0.402 & 0.356 & -0.550 & -0.236 \\
$q=1$ & 0.408 & 0.362 & -0.549 & -0.233 \\
$q=2$ & 0.400 & 0.370 & -0.485 & -0.291 \\
$q=5$ & 0.392 & 0.377 & -0.420 & -0.339 \\
$\theta=(4,1.5)^\top$ & 0.408 & 0.362 & -0.412 &  0.022 \\
$\theta=(4,2.5)^\top$ & 0.408 & 0.362 & -0.549 & -0.233 \\
$\theta=(5,2.5)^\top$ & 0.408 & 0.362 & -0.850 &  -0.387 \\
$\sigma=(0.1,0.3)^\top$ & 0.408 & 0.362 & -0.549 & -0.233 \\
$\sigma=(0.4,0.6)^\top$ & 0.408 & 0.362 & -0.549 & -0.233 \\
$\sigma=(0.8,1.2)^\top$ & 0.408 & 0.362 & -0.549 & -0.233 \\
\bottomrule
\end{tabularx}
\end{table}

\section{Concluding remarks}\label{section conclusion}

In this paper, we established a sufficient SMP for an infinite horizon discounted
optimal control problem with regime switching. By choosing an appropriate discount
factor, we prove the global solvability of infinite horizon forward and
backward SDEs with Markov chains and the asymptotic property of their solutions
at infinity. As an application, we utilize the sufficient SMP to solve
a production planning problem. An explicit expression of the optimal production rate
in a feedback form is derived via an ARE and a linear algebraic equation.

There are several interesting questions that deserve further investigation. Firstly,
in this paper the optimal control problem is formulated to be completely observed.
A more realistic scene is to consider a partially observed setup, for example,
the state process $X$ can be available but the Markov chain $\alpha$ cannot be observed
(see \cite{LXX2023SCL}). Secondly, the production planning problem is essentially
a special kind of LQ problem. It is interesting and, in the meantime more difficult,
to study general infinite horizon LQ optimal control problems with regime switching
and the associated AREs. Thirdly, together with another popular topic, recently
the mean-field control problem with regime switching has attracted great research
interests (see \cite{nguyen2020,nguyen2021}). We also expect to establish an infinite
horizon sufficient SMP within such a framework.

\appendix

\renewcommand{\theequation}{\Alph{section}.\arabic{equation}}

\section{Proof of Lemma \ref{lem:Y}}\label{app:Y-lim}

\begin{proof}
In the structure of a forward SDE, the backward SDE (\ref{BSDE}) can be written as
\begin{equation*}
\begin{aligned}
Y_t=Y_0-\int_0^t[F(Y_s,Z_s,\alpha_s)+\varphi_s]ds+\int_0^t Z_s dW_s+\int_0^t \Lambda _s\bullet dM_s.
\end{aligned}
\end{equation*}
Thus, applying It\^{o}'s formula to $|Y_s e^{Ks}|^2$ over $[0,t]$, we have
\begin{equation*}
\begin{aligned}
|Y_t e^{Kt}|^2
=&|Y_0|^2+\int_0^t e^{2Ks}[2K|Y_s|^2-2\langle Y_s,F(Y_s,Z_s,\alpha_s)+\varphi_s\rangle+|Z_s|^2]ds\\
&+\int_0^t e^{2Ks}|\Lambda_s|^2\bullet d[M]_s+\int_0^t e^{2Ks}\langle Y_s,Z_s\rangle dW_s
+\int_0^t e^{2Ks}\langle Y_s,\Lambda_s\bullet dM_s\rangle,
\end{aligned}
\end{equation*}
which implies that
\begin{equation}\label{Y-FSDE-esti}
\begin{aligned}
&\mathbb{E}\bigg[\sup_{0\leq t\leq T}|Y_t e^{Kt}|^2\bigg]\\
\leq&|Y_0|^2+\int_0^T e^{2Ks}[2K|Y_s|^2+|Z_s|^2]ds+\int_0^T e^{2Ks}|\Lambda_s|^2\bullet d[M]_s\\
&+\int_0^T e^{2Ks}[2|Y_s|^2+|F(Y_s,Z_s,\alpha_s)|^2+|\varphi_s|^2]ds\\
&+\mathbb{E}\bigg[\sup_{0\leq t\leq T}\bigg|\int_0^t e^{2Ks}\langle Y_s,Z_s\rangle dW_s\bigg|\bigg]
+\mathbb{E}\bigg[\sup_{0\leq t\leq T}\bigg|\int_0^t e^{2Ks}\langle Y_s,\Lambda_s\bullet dM_s\rangle\bigg|\bigg].
\end{aligned}
\end{equation}
Note that
\begin{equation}\label{F-Lip}
\begin{aligned}
|F(Y_s,Z_s,\alpha_s)|^2
=&|F(Y_s,Z_s,\alpha_s)-F(0,0,\alpha_s)|^2\\
\leq&(\kappa_{Fy}|Y_s|+\kappa_{Fz}|Z_s|)^2\\
\leq&2\kappa_{Fy}^2|Y_s|^2+2\kappa_{Fz}^2|Z_s|^2.
\end{aligned}
\end{equation}
Combining (\ref{Y-FSDE-esti}) and (\ref{F-Lip}), we have
\begin{equation}\label{Y-FSDE-esti2}
\begin{aligned}
&\mathbb{E}\bigg[\sup_{0\leq t\leq T}|Y_t e^{Kt}|^2\bigg]\\
\leq&|Y_0|^2+\frac{l_0}{2}
+\int_0^T e^{2Ks}[2(K+1+\kappa_{Fy}^{2})|Y_s|^2+(1+2\kappa_{Fz}^{2})|Z_s|^2]ds
+\int_0^T e^{2Ks}|\Lambda_s|^2\bullet d[M]_s\\
&+\mathbb{E}\bigg[\sup_{0\leq t\leq T}\bigg|\int_0^t e^{2Ks}\langle Y_s,Z_s\rangle dW_s\bigg|\bigg]
+\mathbb{E}\bigg[\sup_{0\leq t\leq T}\bigg|\int_0^t e^{2Ks}\langle Y_s,\Lambda_s\bullet dM_s\rangle\bigg|\bigg].
\end{aligned}
\end{equation}
By the BDG inequality and Young's inequality, we obtain
\begin{equation}\label{Y-BDG}
\begin{aligned}
&\mathbb{E}\bigg[\sup_{0\leq t\leq T}\bigg|\int_0^t e^{2Ks}\langle Y_s,Z_s\rangle dW_s\bigg|\bigg]
+\mathbb{E}\bigg[\sup_{0\leq t\leq T}\bigg|\int_0^t e^{2Ks}\langle Y_s,\Lambda_s\bullet dM_s\rangle\bigg|\bigg]\\
\leq&C\mathbb{E}\bigg[\bigg(\int_0^T |e^{2Ks}\langle Y_s,Z_s\rangle|^2ds\bigg)^{\frac{1}{2}}
+\sum_{i,j\in\mathcal{M}}\bigg(\int_0^T|e^{2Ks}\langle Y_s,\Lambda_{ij}(s)\rangle|^2d[M_{ij}](s)\bigg)^{\frac{1}{2}}\bigg]\\
\leq&C\mathbb{E}\bigg[\sup_{0\leq t\leq T}|Y_t e^{Kt}|\bigg(\int_0^T |e^{Ks}Z_s|^2ds\bigg)^{\frac{1}{2}}
+\sum_{i,j\in\mathcal{M}}\sup_{0\leq t\leq T}|Y_t e^{Kt}|\bigg(\int_0^T|e^{Ks}\Lambda_{ij}(s)|^2d[M_{ij}](s)\bigg)^{\frac{1}{2}}\bigg]\\
\leq&\frac{1}{2}\mathbb{E}\bigg[\sup_{0\leq t\leq T}|Y_t e^{Kt}|^2\bigg]
+4C^2\mathbb{E}\bigg[\int_0^Te^{2Ks}|Z_s|^2ds\bigg]+4C^2\mathbb{E}\bigg[\int_0^T e^{2Ks}|\Lambda_s|^2\bullet d[M]_s\bigg].
\end{aligned}
\end{equation}
Substituting (\ref{Y-BDG}) into (\ref{Y-FSDE-esti2}), we have
\begin{equation*}
\begin{aligned}
&\frac{1}{2}\mathbb{E}\bigg[\sup_{0\leq t\leq T}|Y_t e^{Kt}|^2\bigg]\\
\leq&|Y_0|^2+\frac{l_0}{2}+\mathbb{E}\bigg[\int_0^T e^{2Ks}[2(K+1+\kappa_{Fy}^{2})|Y_s|^2
+(1+2\kappa_{Fz}^{2}+4C^2)|Z_s|^2]ds\bigg]\\
&+(1+4C^2)\mathbb{E}\bigg[\int_0^T e^{2Ks}|\Lambda_s|^2\bullet d[M]_s\bigg],
\end{aligned}
\end{equation*}
which yields the estimate (\ref{Y-SK}). Then, by using the same method as Theorem \ref{thm:X-wellpose},
we can show that the map $T\mapsto\mathbb{E}[|Y_T e^{KT}|^2]$ is uniformly continuous. This combined
with $Y\in L_{\mathbb{F}}^{2,K}(0,\infty;\mathbb{R}^n)$ give the asymptotic property (\ref{Y-lim}).
\end{proof}

\section{Proof of Lemma \ref{lem:Riccatisolvable}}\label{app:Riccatisolvable}

\begin{proof}
\textbf{Uniqueness}. We first prove the uniqueness of solution to the ARE \eqref{RE}.
Suppose there exist two non-negative solutions $(\varphi(1),\varphi(2),\ldots,\varphi(m))^\top$
and $(\widetilde{\varphi}(1),\widetilde{\varphi}(2),\ldots,\widetilde{\varphi}(m))^\top$
to \eqref{RE}. By the property $q_{ii}=-\sum_{j\ne i}{q_{ij}}$ and subtracting the two
equations satisfied by $\varphi$ and $\widetilde{\varphi}$, we obtain
\begin{equation}\label{diff}
\frac{\varphi^2(i)-\widetilde{\varphi}^2(i)}{R(i)}
+\bigg(r+\sum_{j\ne i}{q_{ij}}\bigg)[\varphi(i)-\widetilde{\varphi}(i)]
-\sum_{j\ne i}{q_{ij}}[\varphi(j)-\widetilde{\varphi}(j)]=0,\quad i\in\mathcal{M}.
\end{equation}
For simplicity, we denote $\Delta\varphi\in\mathbb{R}^m$ with its $i$-th element
$\Delta\varphi(i)\doteq\varphi(i)-\widetilde{\varphi}(i)$ and
$$
A_{\varphi}
=\left[
\begin{matrix}
\frac{\varphi(1)+\widetilde{\varphi}(1)}{R(1)}+r+\sum_{j\ne 1}{q_{1j}} &	-q_{12}& \cdots & -q_{1m} \\
-q_{11} & \frac{\varphi(2)+\widetilde{\varphi}(2)}{R(2)}+r+\sum_{j\ne 2}{q_{2j}} & \cdots & -q_{2m} \\
\vdots & \vdots & & \vdots \\
-q_{m1} & -q_{m2} & \cdots & \frac{\varphi(m)+\widetilde{\varphi}(m)}{R(m)}+r+\sum_{j\ne m}{q_{mj}} \\
\end{matrix}
\right].
$$
Then, the equation \eqref{diff} can be rewritten in a matrix form as follows:
\begin{equation}\label{Delta}
\begin{aligned}
A_{\varphi}\Delta\varphi=0.
\end{aligned}
\end{equation}
Note that $A_{\varphi}$ is a strictly diagonally dominant matrix, since
$$
\frac{\varphi(i)+\widetilde{\varphi}(i)}{R(i)}+r+\sum_{j\ne i}{q_{ij}}
>\sum_{j\ne i}{|-q_{ij}|}=\sum_{j\ne i}{q_{ij}},\quad i\in\mathcal{M}.
$$
It follows from Horn and Johnson \cite[Theorem 6.1.10]{horn2012matrix} that
$A _{\varphi}$ is an invertible matrix. Therefore, equation \eqref{Delta}
has only zero solution, i.e., $\Delta\varphi(i)=0$, and thus
$\varphi(i)=\widetilde{\varphi}(i)$, $i\in\mathcal{M}$.

To proceed, we denote $b\in\mathbb{R}^m$ with its $i$-th element
$b(i)\doteq(h(i)-\theta(i))\varphi(i)+N(i)c(i)$ and
\begin{equation*}
\begin{aligned}
B_{\varphi}
=\left[
\begin{matrix}
\frac{\varphi(1)}{R(1)}+r+\sum_{j\ne 1}{q_{1j}} & -q_{12} & \cdots & -q_{1m} \\
-q_{11} & \frac{\varphi(2)}{R(2)}+r+\sum_{j\ne 2}{q_{2j}} & \cdots & -q_{2m} \\
\vdots & \vdots & & \vdots \\
-q_{m1} & -q_{m2} & \cdots & \frac{\varphi(2)}{R(2)}+r+\sum_{j\ne m}{q_{mj}} \\
\end{matrix}
\right].
\end{aligned}
\end{equation*}
Similarly, $B_{\varphi}$ is also a strictly diagonally dominant matrix and then an invertible matrix.
Moreover, the equation \eqref{RE-2} can be rewritten in a matrix form as follows:
\begin{equation*}
\begin{aligned}
B_{\varphi}
\left[\begin{array}{c}
\psi(1) \\
\psi(2) \\
\vdots \\
\psi(m) \\
\end{array}
\right]=b,
\end{aligned}
\end{equation*}
which admits a unique solution $B_{\varphi}^{-1}b$.

\textbf{Existence}. Next, we prove the existence of a non-negative solution to \eqref{RE}.
The procedure below to find a solution is essentially based on the so-called
\emph{elimination method}. Without loss of generality, we assume $R(i)\equiv1$,
$i\in\mathcal{M}$. The proof will be similar for the general case.
Then, \eqref{RE} becomes
\begin{equation}\label{REE}
\varphi^2(i)+\bigg(r+\sum_{j\ne i}{q_{ij}}\bigg)\varphi(i)
-\sum_{j\ne i}{q_{ij}}\varphi(j)-N(i)=0,\quad i\in\mathcal{M}.
\end{equation}
We introduce a transformation:
$$
l_i=\frac{1}{2}\bigg(r+\sum_{j\ne i}{q_{ij}}\bigg),
\quad
y_i=\varphi(i)+l_i,
\quad
\widetilde{Q}(i)=N(i)+l_{i}^{2}-\sum_{j\ne i}{q_{ij}}l_j,\quad i\in\mathcal{M},
$$
a function $f_{i}:\mathbb{R}^{m}\mapsto \mathbb{R}$:
\begin{equation}
f_i(y_1,y_2,\ldots,y_m)=y_{i}^{2}-\sum_{j\ne i}{q_{ij}}y_j-\widetilde{Q}(i),\quad i\in\mathcal{M},
\end{equation}
and a number of sets $S_d$, $d=0,1,\ldots,m-1$, defined by:
$$
S_d=\{(y_{d+1},\ldots,y_m)\in\mathbb{R}^{m-d}|y_k\geq l_k,k=d+1,\ldots,m\}.
$$
In order to show the equation \eqref{RE} admits at least one non-negative solution,
it suffices to show the vector-value function $(f_1,\ldots,f_m)$ has at least
one zero point $(y_1^*,\ldots,y_m^*)\in S_0$. In the following, the proof is divided
into three steps.

\emph{Step 1}. For the function $f_1$ (given $(y_2,\ldots,y_m)\in S_1$), we can find a zero point
$$
y_{1}^{*}
=y_{1}^{*}(y_2,\ldots,y_m)
=\sqrt{\sum_{j=2}^{m}q_{1j}y_j+\widetilde{Q}(1)}
=\sqrt{\sum_{j=2}^{m}q_{1j}(y_j-l_j)+N(1)+l_{1}^{2}}\geq l_1.
$$
Clearly, $y_1^*$ is a continuous function of $(y_2,\ldots,y_m)\in S_1$. In addition, by the elementary
inequality $\sqrt{a+b}\leq\sqrt{a}+\sqrt{{b}}$ for any $a,b>0$, there exists a constant $C>0$ such that
\begin{equation}\label{f1est1}
|y_{1}^{*}|\leq C\bigg(1+\sum_{j=2}^m{\sqrt{|y_j|}}\bigg).
\end{equation}
\emph{Step 2}. Substituting $y_1^*$ into the function $f_2$ (given $(y_3,\ldots,y_m)\in S_2$), we denote
$$
u(y_2)\doteq y_{2}^{2}-q_{21}\underset{y_{1}^{*}(y_2,\ldots,y_m)}
{\underbrace{\sqrt{q_{12}y_2+\sum_{j=3}^{m}q_{1j}y_j+\widetilde{Q}(1)}}}
-\sum_{j=3}^{m}q_{2j}y_j-\widetilde{Q}(2).
$$
From the intermediate value theorem, there exists at least one $y_{2}^{*}\geq l_2$ such that
$u(y_{2}^{*})=0$ (noting that $u(l_2)\leq 0$ and $\lim_{y_2\rightarrow\infty}u(y_2)=\infty$
due to the estimate \eqref{f1est1}). Moreover, by Young's inequality
$\sqrt[4]{a}\leq\frac{1}{2C}a+\frac{C^3}{8}$ for any $a,C>0$, we obtain
$$
\begin{aligned}
|y_{2}^{*}|
=&\sqrt{q_{21}\sqrt{q_{12}y_{2}^{*}+\sum_{j=3}^{m}q_{1j}y_j+\widetilde{Q}(1)}+\sum_{j=3}^{m}q_{2j}y_j+\widetilde{Q}(2)}\\
\leq&C\bigg(1+|y_{2}^{*}|^{\frac{1}{4}}+\sum_{j=3}^{m}\sqrt{|y_j|}\bigg)
\leq C\bigg(1+\frac{1}{2C}|y_{2}^{*}|+\frac{C^3}{8}+\sum_{j=3}^{m}\sqrt{|y_j|}\bigg),
\end{aligned}
$$
which leads to that
$$
|y_{2}^{*}|\leq C\bigg(1+\sum_{j=3}^{m}\sqrt{|y_j|}\bigg).
$$
Plugging the above inequality into \eqref{f1est1}, one has
$$
|y_{1}^{*}|\leq C\bigg(1+\sum_{j=3}^{m}\sqrt{|y_j|}\bigg).
$$
So far, we deduce that $y_1^*,y_2^*$ are dependent only on $(y_3,\ldots,y_m)\in S_2$, and
\begin{equation}\label{est:y12}
l_1\leq y_{1}^{*}\leq C\bigg(1+\sum_{j=3}^{m}\sqrt{|y_j|}\bigg),
\quad
l_2\leq y_{2}^{*}\leq C\bigg(1+\sum_{j=3}^{m}\sqrt{|y_j|}\bigg).
\end{equation}
\emph{Step 3}. By applying the elimination method repeatedly, we can obtain
$y_{1}^{*},y_{2}^{*},\ldots,y_{m-1}^{*}$, which are dependent only on $y_m$, and
$$
l_i\leq y_{i}^{*}\leq C\bigg(1+\sqrt{|y_m|}\bigg),\quad i=1,\ldots,m-1.
$$
Substituting $y_{1}^{*},y_{2}^{*},\ldots,y_{m-1}^{*}$ into $f_m$, we have
$$
f_m(y_{1}^{*}(y_m),y_{2}^{*}(y_m),\ldots,y_{m-1}^{*}(y_m),y_m)
=y_{m}^{2}-\sum_{j\ne m}q_{mj}y_{j}^{*}(y_m)-\widetilde{Q}(m)\doteq g(y_m).
$$
We claim that there exists a zero point of $g(y_m)$ in $[l_m,\infty)$,
which is denoted as $y_m^*$. In fact, note that
$$
g(l_m)=l_{m}^{2}-\sum_{j\ne m}q_{mj}y_{j}^{*}(l_m)-\widetilde{Q}(m)\leq l_{m}^{2}-(N(m)+l_{m}^{2})\leq0,
$$
and
$$
\lim_{y\rightarrow\infty}g(y)
=\lim_{y\rightarrow\infty}\bigg(y^2-\sum_{j\ne m}q_{mj}y_{j}^{*}(y)-\widetilde{Q}(m)\bigg)
\geq\lim_{y\rightarrow \infty}\bigg(y^2-C\sum_{j\ne m}q_{mj}(1+\sqrt{y})-\widetilde{Q}(m)\bigg)=\infty.
$$
Hence, from the intermediate value theorem, there exists a $y_m^*\geq l_m$ such that $g(y_m^*)=0$.

Through the above three steps, we finally obtain a zero point $(y_1^*,\ldots,y_m^*)\in S_0$
for the function $(f_1,\ldots,f_m)$, which gives the existence of a non-negative solution
to the ARE (\ref{RE}).
\end{proof}

\section*{Declarations}

\textbf{Competing Interests.}
The authors do not have any competing interests.

\end{document}